\documentclass[final,11pt]{amsart}
\usepackage{amsmath} \usepackage{paralist}
\usepackage[misc]{ifsym}
\usepackage{epsfig} 
\usepackage{epstopdf} 
\usepackage[colorlinks=true]{hyperref}
\hypersetup{urlcolor=blue, citecolor=red}
\allowdisplaybreaks

\newtheorem{theorem}{Theorem}[section]
\newtheorem{corollary}[theorem]{Corollary}

\newtheorem{lemma}[theorem]{Lemma}
\newtheorem{proposition}[theorem]{Proposition}

\theoremstyle{definition}

\newtheorem{example}[theorem]{Example}

\newcommand{\dotp}[2]{\langle #1,\,#2 \rangle}

\newcommand{\R}{{\mathbb R}}

\newcommand{\cH}{\mathcal H}

\newcommand{\argmin}{{\rm argmin}\kern 0.12em}

\newcommand{\ie}{{\it i.e.}\,\,}
\DeclareMathOperator*{\prox}{prox}
\newcommand{\interior}{{\rm int}\kern 0.06em}
\newcommand{\inte}{{\rm int}\kern 0.06em}
\newcommand{\cl}{{\rm cl}\kern 0.06em}
\newcommand{\zer}{{\rm zer}\kern 0.06em}
\newcommand{\gph}{{\rm gph}\kern 0.06em}
\newcommand{\dom}{{\rm dom}\kern 0.06em}
\newcommand{\pr}{{\rm pr}\kern 0.06em}

\newcommand{\n}{{\nabla}}
\newcommand{\p}{{\partial}}
\def\a{\alpha}

\def\b{\beta}
\def\d{\delta}

\def\g{\gamma}
\def\m{\mu}

\def\l{\lambda}
\def\<{\langle}
\def\>{\rangle}
\def\r{\rho}

\def\p{\partial}
\PassOptionsToPackage{normalem}{ulem}
\usepackage{mathtools}
\DeclarePairedDelimiter\abs{\lvert}{\rvert}%
\DeclarePairedDelimiter\norm{\lVert}{\rVert}%

\usepackage{cite}

\usepackage [normalem]{ulem}
\usepackage{soul}
\makeatletter
\let\oldabs\abs
\def\abs{\@ifstar{\oldabs}{\oldabs*}}

\newcommand{\Rb}{\mathbb R\cup\{+\infty\}}

\let\oldnorm\norm
\def\norm{\@ifstar{\oldnorm}{\oldnorm*}}
\makeatother
\usepackage{epsfig}
\usepackage{pict2e}
\usepackage{graphics} 

\if
 {
 \usepackage[pageref]{backref}
\renewcommand*{\backrefalt}[4]{%
\ifcase #1 %
(Not cited)%
\or
(Cited on p.~#2)%
\else
(Cited on pp.~#2)%
\fi
}

}
\fi

\title[Strong Convergence for  a Proximal Point Algorithm]{Combining Strong Convergence, Values Fast Convergence and  Vanishing of Gradients  for  a Proximal Point Algorithm Using Tikhonov Regularization in a Hilbert Space}
\author[A. C. Bagy, Z. Chbani and H. Riahi]{}



\begin{document}
\maketitle
\centerline{\scshape Akram Chahid Bagy$^{{\href{mailto:akram.bagy@gmail.com}{\textrm{\Letter}}}}$, Zaki Chbani$^{{\href{mailto:chbaniz@uca.ac.ma}{\textrm{\Letter}}}}$ and Hassan Riahi$^{{\href{mailto:h-riahi@uca.ac.ma}{\textrm{\Letter}}}}$}
\medskip
{\footnotesize
 \centerline{Faculty of Sciences Semlalia, Mathematics, Cadi Ayyad university}
   \centerline{40000 Marrakech, Morroco}
}

\bigskip


\begin{abstract}
In a real Hilbert space $\cH$. Given any function $f$ convex differentiable whose solution set $\argmin_{\cH}\,f$ is nonempty, by considering the Proximal Algorithm $x_{k+1}=\text{prox}_{\b_k f}(d x_k)$, where $0<d<1$ and $(\b_k)$ is nondecreasing function, and by assuming some assumptions on $(\b_k)$, we will show that the value of the objective function in the sequence generated by our algorithm converges in order $\mathcal{O} \left( \frac{1}{ \beta _k} \right)$  to the global minimum of the objective function, and that the generated sequence converges  strongly to the minimum norm element of $\argmin_{\cH}\,f$, we also obtain a convergence rate of gradient toward zero. Afterwards, we extend these results to non-smooth convex functions with extended real values. 
\end{abstract}

\section{Introduction}
In a Hilbert setting  $\mathcal{H}$, we denote by $\langle\cdot ,\cdot\rangle$ and  $\Vert \cdot\Vert$ the associate inner product and norm respectively. Let $f:\,\cH\rightarrow ]-\infty , +\infty]$ be a  proper lower semicontinuous convex function.
Consider the following  optimization problem 
\begin{equation}\label{minf}
\tag*{$(\mathcal{P})$} \min \left\lbrace f(x):\, x\in \mathcal{H}\right\rbrace ,
\end{equation}
 whose global solution set $S:=   \argmin_{\mathcal{H}} f $ is assume to be  nonempty.

The goal of this paper is to construct a  suitable  sequence $(x_k)$
in order to jointly obtain the fast rate of convergence for  the objective function $f(x_k)$ to the optimal  value $f^\star=\min_{\mathcal{H}}f$, the fast rate norm convergence  to zero of a selected gradient in the subdifferential $\partial f(x_k)$, and also the strong convergence of the iterates $x_k$ to $x^\star$ the element of minimum norm of $S$.
As a guide in our study, we first rely on the asymptotic behaviour  of trajectories of  a
suitable continuous dynamical system.
So, let's first recall some classic results concerning the algorithms associated with the continuous steepest descent.
Given $\lambda >0$, the gradient descent method  
$$x_{k+1}=x_k-\lambda \nabla f (x_k)
$$ 
and the proximal point method 
$$x_{k+1}=\text{prox}_{\lambda f}(x_k)=\mbox{\rm{argmin}}_{ \xi \in \mathcal H}\left( \lambda f (\xi) + \frac{1}{2}\| x -\xi \|^2\right)
$$
are the  basic blocks for algorithms in  convex optimization. By strong convexity of $ \lambda f (\cdot) + \frac{1}{2}\| x -\cdot \|^2$, its minimizer always exists and is unique, so that the  proximal  operator $\text{prox}_{\lambda f}(\cdot)$ is well defined on the hole space $ \mathcal  H$.

By interpreting  $\lambda$  as a fixed time step, these two algorithms can be
respectively  obtained as either the forward (explicit: $\frac{x_{k+1}-x_k}{\lambda}+\n f(x_{k})=0$) or  the  backward (implicit: $\frac{x_{k+1}-x_k}{\lambda}+\n f(x_{k+1})=0$) discretization  of the  continuous system
 \begin{equation}\label{FODE}
 \dot{x}(t) + \nabla f (x(t))= 0.
\end{equation}

This gradient  method goes back to Cauchy (1847),  while
the proximal algorithm was first introduced
by Martinet \cite{martinet} (1970), and then developed   by Rockafellar \cite{Rock} who extended it to solve monotone inclusions. 
One can consult \cite{BaCo},   \cite{PaBo}, \cite{Pey},  \cite{Pey_Sor},
 \cite{Polyak2}, for a recent account on the proximal methods, that play a central role in nonsmooth optimization as a basic block of many splitting algorithms.

\if{
Historically,  
Polyak \cite{Polyak1} considered the Heavy ball system 
\begin{equation}
\ddot{x}(t)+\a\dot{x}(t) +\n f(x(t))=0,
\end{equation}
which physically describes the motion of a material point (ball) rolling over the graph of the function $f$ and subject to a friction proportional to the velocity.
He proved, see \cite{Polyak2}, that in the case where $f$ is $\m$-strongly convex, the generated trajectory converges linearly to the unique minimizer of $f$ with the rate $\mathcal{O}\left( e^{-2\sqrt{\m}t} \right)$ as $t \to +\infty$.

We will rely on the recent developments linking  Nesterov accelerated method for convex optimization with  inertial gradient dynamics.
As a main originality of our approach, we will show that time rescaling of these dynamics  leads to proximal algorithms that converge arbitrarily fast. 
 
Precisely, ${\rm (IPA)}_{\alpha_k, \lambda_k}$  bears close connection with the {\it Inertial Gradient System}
\begin{equation} \label{eq:basic-inertial-flow}
{\rm (IGS)_{\gamma} } \quad \ddot{x}(t)+\gamma(t) \dot{x}(t)+  \nablaf(x(t))=0,
\end{equation}
 which is a non-autonomous second-order differential equation 
where $\gamma(\cdot)$ is a positive viscous damping parameter.
As pointed out by Su-Boyd-Cand\`es in \cite{SBC}, the ${\rm (IGS)_{\gamma} }$ system with $\gamma (t) = \frac{3}{t}$
  can be seen as a continuous version of the accelerated gradient
method of Nesterov (see \cite{Nest1,Nest2}). This method has been developed  to deal with large scale structured convex minimization problems, such as the FISTA algorithm of Beck-Teboulle \cite{BT}.
 These methods guarantee (in the worst case)
the convergence rate $f(x_k)-\min_\cH f= \mathcal O \left(\frac{1}{k^2}\right)$, where $k$ is the number of iterations.
Convergence of the  sequences generated
by FISTA, has not been established so far (except in the one dimensional case, see \cite{ACR-subcrit}). 
}\fi

Recently, research axes have focused on coupling first-order time-gradient systems with a Tikhonov approximation whose coefficient tends asymptotically to zero.
Let us recall that by solving a general  ill posed problem $b= Bx$ in the sense of Hadamard, Tikhonov proposed the new method  which he called a method of "regularization", see \cite{Tikhonov1,Tikhonov2}. This method, which has been developed  in \cite{Morozov,TikhonovLY} and references therein, consists  first in solving the well-posed problem $b=Bx_\epsilon +\epsilon x_\epsilon$, and then in converging $x_\epsilon$ towards a selected point $\bar x$ (as $\epsilon\rightarrow 0$) which verifies $B\bar x=b$.

\if{
 To provide a fast dynamic approach to the hierarchical minimization problem which consists in finding the minimum norm solution of a convex minimization problem,    Attouch-Czarnecki in \cite{attcz17} (see also \cite{attcz17}, \cite{Cabot-inertiel}, \cite{CEG},  \cite{JM-Tikh})  considered the  system
\begin{equation}\label{HBF-Tikh}
 \ddot{x}(t) + \alpha \dot{x}(t) + \nabla f(x(t)) + \varepsilon (t) x(t) =0.
\end{equation}
When the time goes to infinity and the Tikhonov regularization parameter is supposed to tend slowly towards zero, i.e., $\int_0^{+\infty} \varepsilon (t) dt = + \infty$,
 they proved that  any solution $x(\cdot)$ of \eqref{HBF-Tikh} strongly converge towards the  minimum norm element $x^*$ of $\argmin f$. 
\\
In the quest for faster convergence of the  Su, Boyd and
Cand\`es  dynamical system  in \cite{SBC}, the following  Tikhonov regularization  system with asymptotically vanishing damping
\begin{equation}\label{edo001-0}
 \ddot{x}(t) + \frac{\alpha}{t} \dot{x}(t) + \nabla f (x(t)) +\varepsilon(t) x(t)=0,
\end{equation}
was studied by Attouch, Chbani, and Riahi in \cite{ACR}. Indeed, the case $\epsilon (\cdot)\equiv 0$ and $\alpha =3$ of Su, Boyd and
Cand\`es provides
 a continuous version of the  Nesterov accelerated gradient method, see  \cite{Nest1,Nest2,Nest3,Nest4}.
\\
According to the structure of the heavy ball method for strongly convex functions, the viscous damping coefficient is proportional to the square root of the Tikhonov regularization parameter.
To solve this question, Attouch and L\' aszl\' o in \cite{al} 
 considered  the following system
\begin{equation}\label{system JDE}
\ddot{x}(t)+  \alpha\sqrt{ \varepsilon (t) } \dot{x}(t)+\nabla f(x(t)) +  \varepsilon (t)  x(t) =0.
\end{equation}
In \cite{ABCR-JDE}, Attouch et al.   succeeded in obtaining $\lim_{t\to\infty}\|x(t)-x^*\|= 0$ in the case where $\alpha >0$ and  $ \varepsilon (t)=  \frac{1}{t^{2r}}$ with  $0<r<1$.
This approach is developed by  relying on the proximal operator 
$$ x_{\epsilon (t)} = \hbox{prox}_{\frac1{\epsilon (t)}f}(0):=\underset{\mathcal{H}} { \argmin } { \lbrace f(x)+\frac{\epsilon (t)}{2}} \norm{x}^2 \rbrace .
$$
They obtained, when $t \rightarrow +\infty $, the following convergence rates 
\begin{equation}\label{abcrJDE}
f(x(t))-\underset{\mathcal{H}}{\min}f = \mathcal{O} \left( \dfrac{1}{t^{2r}} \right),\;\;
\norm{ x(t)- x_{\epsilon (t)} }^2   = \mathcal{O} \left( \dfrac{1}{t^{1-r}} \right) \;\text{ and }\;\norm{ \dot{x}(t) }^2   = \mathcal{O} \left( \dfrac{1}{t^{r+1}} \right);
\end{equation}
and mainly, the strong convergence of $x(t)$ to the minimal norm solution  of  \ref{minf}
\\
Even more interestingly, by applying a rescaling and averaging method introduced by Attouch, B\c ot and Nguyen \cite{ABotNguyen}, reference \cite{ACR3} improves the presentation of the  evolution system \eqref{system JDE} by returning to the implicit Hessian form 
 \begin{equation}\label{1st-damped-id05}
\ddot x(t) + \left( \frac{\alpha}{t}  +\dfrac{2^r (\alpha - 1)^{r-1}}{t^{2r-1}}  \right)  \mathring x(t) + \nabla f\left(x(t)+\dfrac{t}{\alpha - 1}\mathring x(t)\right) + 
\dfrac{2^r(\alpha - 1)^r}{t^{2r}} x(t)=0,
\end{equation}
where $0<r\leq 1, \alpha >1$ and $\delta>1$
and then \cite[Theorems 6, 7]{ACR3} complete and improve the previous results.  These recent results provide, also for $r=1$, similar rapid convergence of values, strong convergence to the minimum norm minimizer  and again fast convergence of the gradients to zero.
\\
In the context of non-autonomous dissipative dynamic systems, reparameterization in time is a simple and universal
means to accelerate the convergence of trajectories. This is where the coefficient
$\beta(t)$ comes in as a factor of $\nabla f(x(t))$:
\begin{equation}\label{equation2}
\ddot{x}(t) +  \frac{\alpha}{t} \, \dot{x}(t) + \beta (t) \nabla f(x(t))    =0,
\end{equation}
 In  \cite{ACR1} and \cite{ACR2},
the authors  proved that under appropriate conditions on $\alpha$ and $\beta(t)$, 
$f(x(t)) -\min f =  \mathcal{O} \left( \frac1{t^2\beta(t)} \right)$; hence an improvement of the convergence rate  for the values
is reached by taking $\beta (t)\rightarrow +\infty$ as $t\rightarrow +\infty$. 
\\
In a later paper  \cite{BCR1}, we considered a similar system as \eqref{equation2} without Nesterov's acceleration parameter $\frac1t$ and with the Tikhonov regularizing term $\frac{c}{\beta (t)}\|\cdot\|^2$ to the convex function $f$:
\begin{equation}\label{equation}
 \ddot{x}(t) +  \alpha \, \dot{x}(t)+  + \beta (t) \nabla \varphi _t (x(t))      =0,
\end{equation}
where $ \varphi _t(x):=f(x)+ \dfrac{c}{2\beta (t)}\norm{x} ^2 $ and $\alpha , c, \d$ are three positive real numbers.\\
Under suitable assumptions on $\b(t)$ we have obtained, when $t \to +\infty$, convergence of the values to the global minimum of $f$, strong convergence of the trajectories to the minimum norm solution, and strong convergence of the velocity to zero, with the following rates:
\begin{equation*}
f(x(t))-\underset{\mathcal{H}}{\min}f = \mathcal{O} \left( \dfrac{1}{ \beta (t)} \right), \;\;\;\;
\norm{ x(t)-  x_{t} }^2   = \mathcal{O} \left(    \dfrac{\dot{\beta}(t)}{\beta (t)}+e^{-\mu t} \right)
\end{equation*}
and 
\begin{equation*}
 \norm{\dot{x}(t)}^2 =   \mathcal{O}\left( \dfrac{\dot{\beta}(t)}{\beta (t)} + e^{-\mu t}  \right),
\end{equation*}
where $ x_{t} =\hbox{prox}_{\frac{\beta (t)}{c}f}(x):=\underset{\mathcal{H}} { \argmin } { \lbrace f(x)+\frac{c}{2\b (t)}} \norm{x}^2 \rbrace$, and $\m <\frac{\a}{2}$.\\
}\fi

The minimization of the function $ \varphi _t(x):=f(x)+ \dfrac{c}{2\beta (t)}\norm{x} ^2 $, where  $c$ is a positive real number and $\beta (t) $ goes to $+\infty$ as $t\rightarrow +\infty$, can be seen as a penalization of the problem of minimizing  the objective function $\frac12\|\cdot\|^2$ under the constraint $x\in \text{argmin} f$. This is also a two level hierarchical minimization problem, see \cite{cabo05,attcz10,attcz17,ccr00}. 
\\
Knowing that $ \frac{c}{2}\norm{\cdot} ^2+\beta (t)f(\cdot)$ cross towards $\frac{c}{2}\norm{\cdot} ^2+\iota_{\text{argmin} f}(\cdot)$ as $t\rightarrow+\infty$ if $\iota_{C}$ is the indicator function of the set $C$, i.e., $\iota_{C} (x) = 0$ for $x\in C$, and $+\infty$ outwards. This monotone convergence is proved as a variational convergence  and then as $t\rightarrow+\infty$ (see \cite[Theorems 3.20, 3.66]{att83}) the corresponding weak$\times$strong or strong $\times$weak graph convergence of the associated subdifferential operators: $\nabla (\frac{c}{2}\norm{\cdot} ^2+\beta (t)f){\displaystyle\longrightarrow^G} \partial(\frac{c}{2}\norm{\cdot} ^2+\iota_{\text{argmin} f})$, the convex subdifferential associated to the  proper convex lower semicontinuous function $\frac{c}{2}\norm{\cdot} ^2+\iota_{\text{argmin} f}(\cdot)$.  As $t\rightarrow +\infty$, suppose that the unique minimizer $z_t$ of $\frac{c}{2}\norm{\cdot} ^2+\beta (t)f$ weakly converges to some $\bar z$, then $(z_t,0)$ weak$\times$strong converges to $(\bar z,0)$ in the graph of $\partial(\frac{c}{2}\norm{\cdot} ^2+\iota_{\text{argmin} f}) $; this can be explained as  
$$
0\in  \partial\left(\dfrac{c}{2}\norm{\cdot} ^2+\iota_{\text{argmin} f}\right)(\bar z) = c\bar z +  \partial(\iota_{\text{argmin} f})(\bar z).
$$
 The final equality is due to continuity of the convex function $\frac{c}{2}\norm{\cdot} ^2$ at some point in the nonempty set ${\text{argmin} f}$, which is the effective domaine of the  convex lower semicontinuous function $\iota_{\text{argmin} f}$.

As asymptotical behaviour, for $t\rightarrow+\infty$, this can be explained as the steepest descent dynamical system 
\begin{equation}\label{systeme continue}
0\in \dot{x}(t)  +\partial\left(\dfrac{c}{2}\norm{\cdot} ^2+\beta (t)f\right)(x(t))=  \dot{x}(t)  +cx(t) +\beta (t)\partial f(x(t)).
\end{equation}
 An abundant literature has been devoted to the asymptotic hierarchical minimization property which results from the introduction of a vanishing viscosity term (in our context the Tikhonov approximation) in gradient-like dynamics. For first-order gradient systems and subdifferential inclusions, see \cite{Att2,AttCom}. In \cite{AttCom},  Attouch and Cominetti coupled the dynamic steepest descent method and a Tikhonov regularization term
$$
\dot x(t)+\partial f(x(t))+\epsilon (t)x(t)\ni 0.
$$
   The striking point of their analysis is the strong convergence of the trajectory $x(t)$ when the regularization parameter $\epsilon (t)$ tends to zero with a sufficiently slow rate of convergence $\epsilon \not\in L^1(\mathbb R_+, \mathbb R)$. Then the strong limit is the minimum norm element of  $\text{agrmin }f$. However, if $\epsilon (t) = 0$ we can only expect a weak convergence of the induced trajectory $x(t)$.
 Attouch and Czarnecki in \cite[Theorem 3.1]{attcz10} studied the asymptotic behaviour, as time variable $t$ goes to $+\infty$, of a general nonautonomous first order dynamical system 
$$
0\in \dot{x}(t)  +\partial \varphi (x(t)) + \beta(t)\partial f(x(t)),
$$
and proved weak convergence of $x(t)$ to some $\bar x$ in $\text{argmin}\varphi $ on ${\text{argmin} f}$ that satisfies $0\in \partial(\varphi +\iota_{\text{argmin} f})(\bar x)$. This can be translated for $\varphi = \frac12\|\cdot\|^2$ 
 to $\bar x$ is the minimal norm solution  of  \ref{minf}. This can be considered as a combination of two techniques: the time scaling of a damped inertial gradient system (see \cite{ACR1, ACR2, ACRA}), and the Tikhonov regularization of such systems (see \cite{BCR1, ABCR-JDE} and related references).\\

Our approach in this paper is first to derive,  via Lyapunov analysis, convergence results ensuring  fast convergence of values,  fast convergence of gradients towards zero and strong convergence towards the minimum norm element of $S = \text{argmin}_{\cH} f$, of solution of  the Cauchy type dynamical system 
\begin{equation}\label{systeme continue}
\dot{x}(t)+\b(t) \n f(x(t))+c x(t) =0.
\end{equation}
   This makes it possible to emerge simpler and less technical proofs and thus to better schematize the proof of the results associated with the algorithms generated by the proposed discretizations.

\if{
The Proximal Point Algorithm (PPA) was introduced by Martinet \cite{Martinet} for resolving the monotone inclusion $0\in Ax$. The algorithm  (PPA)  is given by $x_{k+1}=\text{prox}_{\l_kA}( x_k)=(Id+\l_k A)^{-1}(x_k)$.
Rockafeller's \cite{Rockafellar} showed that the (PPA) converges weakly to some point in $A^{-1}(0) $ whenever $A^{-1}(0) \neq \emptyset$ and $(\l_k)$ is bounded away from zero. After that he posed the open question of whether  (PPA) converges strongly or not.  Güler \cite{Guler} answers that question by constructing an example in which the sequence generated by (PPA) converges weakly but not strongly. 
Lehdili and Moudafi \cite{Lahdili} combined the  proximal algorithm and Tikhonov regularization to propose  $x_{k+1}=\text{prox}_{\l_k(A+t_kId)}( x_k)$. This algorithm converges strongly to
a point in $A^{-1}(0) $.
As a result, much several authors proposed modifications of (PPA) to improve weak convergence to strong convergence, see for examples \cite{Kam-Ta, Solodov , Sow, Wang}.
}\fi

So, to attain a solution of the problem $(\mathcal{P})$ for nonsmooth convex function $f$, we consider the following implicit  discretization of the set-valued system 
$
0\in \dot{x}(t)+\b(t) \partial f(x(t))+c x(t) ,
$
 with $\beta(k)=\frac{\b_k}{d}$ and $c=\frac{1-d}d\in ]0,1[$:
\begin{equation}\label{disc syst }
x_{k+1}-x_k+\dfrac{\b_k}{d} \p f(x_{k+1})+\dfrac{1-d}{d}x_{k+1}\ni 0.
\end{equation}
Recall that the proximal operator can be formulated as follows $\text{prox}_{\b f}( x):=\left( I+\b \p f \right)^{-1}(x)$, then iteration \eqref{disc syst } can be reformulated as 
the proximal algorithm:
\begin{equation}\label{proximal algorithm}
x_{k+1}=\text{prox}_{\b_k f}(d x_k).
\end{equation}
Thus, we adhere to  establish similar proposals as fast convergence of values and  gradients towards zero and strong convergence   of the proximal algorithm \eqref{proximal algorithm} to $x^*$ under  suitable  conditions on the sequence $(\b_k).$
\if{
Recall that $\text{prox}_{\b f}( x):=\left( I+\b \p f \right)^{-1}(x)$, then algorithm \eqref{proximal algorithm} can be reformulated as 
\begin{equation}\label{disc syst }
x_{k+1}-x_k+\dfrac{\b_k}{d} \p f(x_{k+1})+\dfrac{1-d}{d}x_{k+1}\ni 0,
\end{equation}
Thus algorithm \eqref{proximal algorithm}, as mentioned above,  may be interpreted as an implicit  discretization of the system \eqref{systeme continue} with $\beta(k)=\dfrac{\b_k}{d}$ and $c=\dfrac{1-d}d>0$.
}\fi \\

The organization of the rest of the paper is as follows. In second section  we recall basic facts concerning Tikhonov approximation.  We also show  the main result of this work concerning the strong convergence property of algorithm \eqref{proximal algorithm}. In section 3, we apply these results to the two particular cases of the sequence $(\b_k)$. Finally, the last section is devoted to the extension of these results to nonsmooth convex functions.

\section{Convergence rates for continuous case}
Return to the differential equation \eqref{systeme continue} :
\begin{equation*}
\dot{x}(t)+\b(t) \n f(x(t))+c x(t) =0.
\end{equation*}
We suppose the following conditions:
\begin{equation*}\label{H01}
\tag*{$(\mathbf{H}_f)$}\left \{
\begin{array}{lll}
(i)&& f \textit{ is convex and  differentiable on } \cH,\\\\
(ii)&& S := \argmin \, f \neq \emptyset\, , \\\\
(iii)&& \nabla f  \textit{  Lipschitz continuous on bounded subsets of }\mathcal{H},
\end{array}
\right.
\end{equation*}
and, there exists $\mu >0$ such that, for $t> t_0$
\begin{equation*}\label{H0}
\tag*{$(\mathbf{H}_\beta )$} \;
\left \{
\begin{array}{lll}
(i)&& \beta (t) \text{ is a positive,}\; \mathcal{C}^2\, \text{ function with }  \dot{\beta}(t)\neq 0,\\\\
(ii)&&
\max\left(\dfrac{\dot{\b}(t)}{\b(t)}\;,\;1\right)  \leq c-\m  ,\\
(iii)&& 
{	\underset{t\rightarrow +\infty}{\limsup} \,\dfrac{  1 + \dfrac{\dot{\beta } (t)}{\beta  (t)}}{\mu + \dfrac{\ddot{\beta } (t)}{{\dot\beta} (t)} - \dfrac{\dot{\b}(t)}{\b(t)}} <+\infty  .}
\end{array}
\right.\hspace{3cm}
\end{equation*}
Let us denote by $x^*$  the minimum norm element  of  $S$, and introduce the energy function $E(t)$  that  is defined on $[t_0, +\infty[$ by
\begin{equation}\label{3}
E(t) := 
\b(t)\left(\varphi_{t}(x(t))-\varphi_{t}(y(t))\right) +\dfrac{c}{2}\|x(t)-y(t)\|^{2}
\end{equation}
where $x(t)$ is a solution of \eqref{systeme continue}, $y(t) = \argmin_{\cH}{\varphi}_{t}$    and
$
\varphi_t (x) := f(x) + \dfrac{c}{2\b(t)} \|x\|^2.
$\\
We remark that \eqref{systeme continue} becomes 
\begin{equation}\label{systeme continue1}
\dot{x}(t)+\b(t) \n \varphi_{t}(x(t)) =0.
\end{equation}
and, under the initial condition $x(t_0)=x_0\in \cH$, it admits a unique solution. 

\begin{theorem}\label{Th1}
Let $f: \mathcal{H} \rightarrow \mathbb{R}$ be a convex function satisfying condition \ref{H01} and $x(.): [t_0 , +\infty[ \rightarrow \mathcal{H}  $ be a solution of the system \eqref{systeme continue}.\\
If $\beta(t)$ satisfies \ref{H0}, then $x(t)$ strongly converges to $x^*$, and there exists $t_1 \geq t_0$ such that, for  $t \geq t_1$  
\if{
\begin{equation*}\label{estim E}
		E(t) \le  \frac{ E(t_1) \gamma(t_1)}{\gamma(t)} + \frac{(1+M)\| x^* \|^2 }{2\gamma(t)} \int_{t_1}^{t} 
		 \dfrac{c(s)\dot{\beta } (s)}{\beta  (s)} \gamma(s)ds
\end{equation*}
and consequently
}\fi
\begin{eqnarray}
\label{eq:8-1}&&f(x(t)) - \min_{\mathcal H}f = \mathcal{O}\left( \dfrac{1}{\b (t)}\right);\\
\label{eq:8-2}&& \|x(t) - y(t)\|^2 			= \mathcal{O}\left( \frac{ 1}{e^{\mu t}} + \dfrac{\dot{\b}(t)}{\b(t)}\right);\\
\label{eq:8-3}&&\norm{\nabla f(x(t))}^2 	= \mathcal{O}\left( \dfrac{1}{\b^(t)}\right).
\end{eqnarray}
\end{theorem}
%

\begin{proof}
We have defined $E(t)$ in \eqref{3} by
$E(t)=A(t) + B(t)$ where 
$$
A(t)=\beta (t) \left( \varphi _t (x(t))- \varphi _t (y(t)) \right)   \text{and } B(t)=\dfrac{c}{2} \norm {x(t)-y(t)}^2.
$$
So, we need to fit a first  order differential inequality on $E(t)$ in order to be able to get the desired upper bound 
\begin{equation}\label{estim E}
		E(t) \le  \frac{ E(t_1) \gamma(t_1)}{\gamma(t)} + \frac{(1+M)\| x^* \|^2 }{2\gamma(t)} \int_{t_1}^{t} 
		 \dfrac{c(s)\dot{\beta } (s)}{\beta  (s)} \gamma(s)ds.
\end{equation}
 To calculate the derivative of the functions  $A(t)$ and $B(t)$, we remark that the mapping $t\mapsto y (t)$ is absolutely continuous (indeed locally Lipschitz), see \cite[section VIII.2]{Brezis2}. We conclude that $y(t)$ is  almost everywhere  differentiable on $[t_0,+\infty[$, and then using differentiability of $f$, we also deduce  almost everywhere  differentiability of $A(t)$ and $B(t)$ on the same intervalle.
Thus, we have
$$
\dot{A}(t)=\dot{\beta} (t)\left( \varphi _t (x(t))- \varphi _t (y(t)) \right) + \beta (t) \left(  \dfrac{d}{dt} \left(\varphi _t (x(t)) \right) - \dfrac{d}{dt} \left(\varphi _t (y(t)) \right) \right) .
$$
Using the derivative calculation and \eqref{systeme continue1}, we remark that
\begin{equation*}
\begin{array}{lll}
\dfrac{d}{dt} \left(\varphi _t (x(t)) \right) &=& \dfrac{d}{dt}\left( f(x(t))+\dfrac{c}{2\beta (t)}\norm{x(t)}^2 \right) \\\\
&=& \large\langle \dot{x}(t)\, , \, \nabla f(x(t))  \large\rangle - \dfrac{c\dot{\beta } (t)}{2\beta ^2 (t)} \norm{x(t)}^2 +\dfrac{c}{\beta (t)} \large\langle \dot{x}(t) \, , \, x(t) \large\rangle \\\\
&=& \large\langle \dot{x}(t)\, , \, \nabla f(x(t))+ \dfrac{c}{\beta (t)} x(t)  \large\rangle - \dfrac{c\dot{\beta } (t)}{2\beta ^2 (t)} \norm{x(t)}^2  \\\\
&=& \large\langle \dot{x}(t)\, , \, \nabla \varphi _t (x(t))  \large\rangle- \dfrac{c\dot{\beta } (t)}{2\beta ^2 (t)} \norm{x(t)}^2 \\
&=& -\beta(t)\left\|  \nabla \varphi _t (x(t))  \right\|^2- \dfrac{c\dot{\beta } (t)}{2\beta ^2 (t)} \norm{x(t)}^2  
\end{array}
\end{equation*}
On the other hand, according to \cite[Lemma 2(i)]{ABCR-JDE}, we have
 $$ 
 \dfrac{d}{dt} \left(\varphi _t (y (t)) \right) =  \dfrac{-c\dot{\beta}(t)}{2\beta ^2 (t)}\norm{y(t)}^2 .
 $$
Thus
\begin{equation}\label{16}
\dot{A}(t)= \dot{\beta} (t)\left( \varphi _t (x(t))- \varphi _t (y(t)) \right)   -
\beta^2 (t) \left[  \left\|  \nabla \varphi _t (x(t))  \right\|^2 + \dfrac{c\dot{\beta } (t)}{2\beta ^3 (t)}\left(  \norm{x(t)}^2- \norm{y(t)}^2 \right) \right] .
\end{equation}
\if{
Since $  \dfrac{-c\dot{\beta} (t)}{2\beta ^2 (t)} \norm{x(t)}^2\leq 0$, we conclude
\begin{equation}\label{5}
\dot{A}(t)\leq \dot{\beta} (t)\left( \varphi _t (x(t))- \varphi _t (y(t)) \right) +  \beta (t)  \large\langle \dot{x}(t)\, , \, \nabla \varphi _t (x(t))  \large\rangle  + \dfrac{c\dot{\beta}(t)}{2\beta  (t)}  \norm{y(t)}^2 .
\end{equation}
}\fi
We also have
\begin{equation*}
\begin{array}{lll}
\dot{B} (t) &=&   c\large\langle x(t) - y(t)  \, , \,  \dot{x}(t)- \dot y(t)  \large\rangle\\
&=&  - c\large\langle x(t) - y(t)  \, , \, \beta (t) \nabla \varphi _t (x(t)) + \dot y(t) \large\rangle .
\end{array}
\end{equation*}
\if{
Thus
\begin{equation}\label{6}
\begin{array}{lll}
\dot{B}(t) &=& (\tau ^2 - \alpha \tau ) \large\langle x(t) - y(t) \, , \, \dot{x}(t) \large\rangle +(\tau - \alpha ) \norm {\dot{x}(t) }^2 \\
\end{array}
\end{equation}
According to Condition $(\mathbf{H}_1)$ and $\underset{a \longrightarrow +\infty}{\lim}\, \dfrac{a-1}{a+1} =  1 $,    one can choose $a>1$ large enough and  $t_1 \geq t_0$ such that for all $t \geq t_1$
\begin{equation}\label{existance de a}
 \dfrac{\dot{\beta} (t)}{ \beta (t)} \leq \dfrac{a-1}{a+1} \alpha  .
\end{equation}
}\fi
Noting that for every $u,v\in \mathcal{H},\, \lambda >0$, $\langle u,v\rangle\leq \frac{\lambda}2\| u\|^2+\frac1{2\lambda}\|v\|^2$,  we get
\begin{equation*} 
- c \large\langle  x(t) - y(t) \, , \, \dot y(t) \large\rangle \leq \dfrac{c}{2}\left( \norm{ x(t) - y(t)}^2 + \norm{ \dot y(t)}^2 \right).
\end{equation*}
Using  \cite[Lemma 2(ii)]{ABCR-JDE}, we  obtain  $ \norm{ \dot y(t)} \leq \dfrac{\dot{\beta}(t)}{\beta(t)} \norm{y (t)}\leq \dfrac{\dot{\beta}(t)}{\beta(t)} \norm{x^*}$; we conclude 
\begin{equation}\label{8}
- c \large\langle  x(t) - y(t) \, , \, \dot y(t) \large\rangle \leq \dfrac{c}{2}\left( \norm{ x(t) - y(t)}^2 +  \dfrac{\dot{\beta}^2(t)}{\beta^2(t)} \norm{x^*}^2 \right) .
\end{equation}
Observe that for $c>0$ the function  $\varphi _t$ is $\frac{c}{\b(t)}$-strongly convex,  then we have
%
\begin{equation*}
 - c \beta (t)\large\langle \nabla \varphi _t (x(t))  \, , \,   x(t) - y(t)  \large\rangle  \leq  -  c \beta (t) \left( \varphi _t (x(t))-  \varphi _t \left(y(t)\right) \right) - \dfrac{c^2}{2} \norm{  x(t) - y(t) }^2 .
\end{equation*}
Thus
\begin{equation}\label{9}
\begin{array}{lll}
\dot{B}(t) &\leq&    \dfrac{c}{2}\left( \norm{ x(t) - y(t)}^2 +  \dfrac{\dot{\beta}^2(t)}{\beta^2(t)} \norm{x^*}^2 \right) \\
	&&-  c \beta (t) \left( \varphi _t (x(t))-  \varphi _t \left(y(t)\right) \right) - \dfrac{c^2}{2} \norm{  x(t) - y(t) }^2 \\
&=&  \dfrac{(1-c)c}{2} \norm {x(t)-y(t)}^2  + \dfrac{c\dot{\beta}^2(t)}{2\beta^2(t)} \norm{x^*}^2  -  c \beta (t) \left( \varphi _t (x(t))-  \varphi _t \left(y(t)\right) \right) .
\end{array}
\end{equation}
%
Due to the inequalities \eqref{16} and \eqref{9}, we get
\begin{equation}\label{10}
\begin{array}{lll}
\dot{E}(t) &\leq &  \dot{\beta} (t)\left( \varphi _t (x(t))- \varphi _t (y(t)) \right)   -
\beta^2 (t) \left[  \left\|  \nabla \varphi _t (x(t))  \right\|^2 + \dfrac{c\dot{\beta } (t)}{2\beta ^3 (t)}\left(  \norm{x(t)}^2- \norm{y(t)}^2 \right) \right] \\
	& & +  \dfrac{(1-c)c}{2} \norm {x(t)-y(t)}^2  + \dfrac{c\dot{\beta}^2(t)}{2\beta^2(t)} \norm{x^*}^2  -  c \beta (t) \left( \varphi _t (x(t))-  \varphi _t \left(y(t)\right) \right)  \\
	& = & \left(  \dot{\beta} (t)-c\beta (t)\right) \left( \varphi _t (x(t))- \varphi _t (y(t)) \right)   -
\beta^2 (t)  \left\|  \nabla \varphi _t (x(t))  \right\|^2 \\\\
&\,\,&- \dfrac{c\dot{\beta } (t)}{2\beta  (t)}\left(  \norm{x(t)}^2- \norm{y(t)}^2 \right)  +  \dfrac{(1-c)c}{2} \norm {x(t)-y(t)}^2  + \dfrac{c\dot{\beta}^2(t)}{2\beta^2(t)} \norm{x^*}^2 \\
\end{array}
\end{equation}
and thus,  for all $t \geq t_0$ we have
$$
\begin{array}{lll}
\dfrac{d}{dt} \left[e^{\mu t} E(t) \right] &=& \left( \dot{E}(t)+ \mu E(t) \right)e^{\mu t} \\
	&\leq&  \left[  \left(  \dot{\beta} (t)-(c-\m)\beta (t)\right) \left( \varphi _t (x(t))- \varphi _t (y(t)) \right)   -
\beta^2 (t)  \left\|  \nabla \varphi _t (x(t))  \right\|^2 \right.\\
&\,\,&
- \dfrac{c\dot{\beta } (t)}{2\beta  (t)}\left(  \norm{x(t)}^2- \norm{y(t)}^2 \right)   +  \dfrac{(1+\mu-c)c}{2} \norm {x(t)-y(t)}^2 \\
&& \left.  + \dfrac{c\dot{\beta}^2(t)}{2\beta^2(t)} \norm{x^*}^2   \right] e^{\mu t} .
\end{array}
$$
Using assumption \ref{H0} (ii), we conclude that  for $t\geq t_0$ large enough
\begin{equation}\label{11}
\dfrac{d}{dt} \left[e^{\mu t} E(t) \right] \leq   \dfrac{c\dot{\beta } (t)}{2\beta  (t)}\left(   \norm{y(t)}^2 + \dfrac{\dot{\beta } (t)}{\beta  (t)} \norm{x^*}^2\right)  e^{\mu t} .
\end{equation}
Since $\; \norm{y(t)} \leq \norm{x^*}\;  $  (see \cite[Lemma 2(ii)]{ABCR-JDE} ), we conclude
\begin{equation}\label{13}
\dfrac{d}{dt} \left[e^{\mu t} E(t) \right] \leq   \dfrac{c\dot{\beta } (t)}{2\beta  (t)}\left(  1 + \dfrac{\dot{\beta } (t)}{\beta  (t)}\right)  e^{\mu t} \norm{x^*}^2.
\end{equation}
	
	Now return to condition \ref{H0}(iii), which means
	$$
	\underset{t\rightarrow +\infty}{\limsup} \,\dfrac{\dfrac{\dot{\beta } (t)}{\beta  (t)} \left(  1 + \dfrac{\dot{\beta } (t)}{\beta  (t)}\right)e^{\mu t } }{ \dfrac{d}{dt}\left(\dfrac{\dot{\b}(t)}{\b(t)}e^{\mu t }\right)} <+\infty,
	$$
we obtain existence of $t_1\geq t_0$ and $M>0$ such that for each $t\geq t_1$
	\begin{equation} \label{e14}
	\dfrac{d}{dt} \left[e^{\mu t} E(t) \right] \leq \frac{cM \norm{x^*}^2 }{2} \dfrac{d}{dt}\left(\dfrac{\dot{\b}(t)}{\b(t)}e^{\mu t }\right) .
	\end{equation}
	By integrating on $[t_1,t]$, we get  
	\begin{eqnarray}
		E(t) &\le&  \frac{ E(t_1)e^{\mu t_1}}{e^{\mu t}} + \frac{cM\| x^* \|^2 }{2e^{\mu t}} \int_{t_1}^{t} 
		 \left(\dfrac{\dot{\b}(s)}{\b(s)}e^{\mu s }\right)ds \nonumber\\
		 &\leq&  \frac{ E(t_1)e^{\mu t_1}}{e^{\mu t}} + \frac{c M\norm{x^*}^2 }{2} \dfrac{\dot{\b}(t)}{\b(t)} .   \label{Ee1}
	\end{eqnarray}
 According to the definition of $ \varphi_{t}$,  we have 
 \begin{equation*}
\begin{array}{lll}
f(x(t)) &-& \min_{\mathcal H}f	=  \varphi_{t}(x(t))-\varphi_{t}(x^*)+\dfrac{c}{2\b (t)}\left(\|x^*\|^{2}-\|x(t)\|^{2}\right)  \\ 
	& = & \left[\varphi_{t}(x(t))-\varphi_{t}(y(t))\right]+ [\underbrace{\varphi_{t}(y(t))-\varphi_{t}(x^*)}_{\leq 0} ]+\dfrac{c}{2\b (t)}\left(\|x^*\|^{2}-\|x(t)\|^{2}\right)\\
	& \leq  &\varphi_{t}(x(t))-\varphi_{t}(y(t))+\dfrac{c}{2\b (t)}\|x^*\|^{2}.
\end{array}
\end{equation*}
By definition of $E(t)$,\;  
$$
\varphi_{t}(x(t))-\varphi_{t}(y(t)) \leq \dfrac{E_(t)}{\b (t)},
$$
 which, combined with the above inequality and \eqref{Ee1}, gives 
 \begin{equation}\label{eq:17a}
f(x(t)) - \min_{\mathcal H}f \leq \dfrac{1}{\b (t)}\left(\frac{ E(t_1)e^{\mu t_1}}{e^{\mu t}} + \frac{c \norm{x^*}^2 }{2} \dfrac{\dot{\b}(t)}{\b(t)}+\frac{c}{2}\|x^*\|^{2}\right). 
\end{equation}
This means, according to \ref{H0}(ii), that 
\if{ $for $t$ large enough 
$$f(x(t)) - \min_{\mathcal H}f  =\mathcal{O}\left( \dfrac{1}{\b (t)}\right).$$
}\fi
$f$ satisfies \eqref{eq:8-1}.

Also, 
\if{
by the strong convexity of $\varphi_{t}$, and the realization of the unique minimum of $\varphi_{t}$ at $y(t)$,  we have
 \begin{equation}\label{eq:17}
\varphi_{t}(x(t))-\varphi_{t}(y(t)) \geq \frac{c}{2\b (t)}  \|x(t) - y(t)\|^2 .
\end{equation}
By combining \eqref{eq:17} with $  \varphi_{t}(x(t))-\varphi_{t}(y(t)) \leq \dfrac{E_(t)}{\b (t)}$,
}\fi
using the definition of $E(t)$, we get
$$
E(t) \geq \frac{c}{2}  \|x(t) - y(t)\|^2,
$$
which gives 
 \begin{equation*}
 \|x(t) - y(t)\|^2 \leq \frac{2}{c}  E(t)  \leq  \frac{2}{c}\left(\frac{ E(t_1)e^{\mu t_1}}{e^{\mu t}} + \frac{c \norm{x^*}^2 }{2} \dfrac{\dot{\b}(t)}{\b(t)}\right),
\end{equation*}
and therefore \eqref{eq:8-2} is satisfied.

Return to \ref{H0}(ii), we justify $(x(t))$ is bounded, and then combining \eqref{eq:17a} and Lemma \ref{ext_descent_lemma} we deduce  existence of $L>0$ such that for $t$ large enough
\begin{eqnarray*}
 \| \nabla f (x)\|^2  &\leq& 2L(f(x) - \min_{\cH} f) \leq  \dfrac{2L}{\b (t)}\left(\frac{ E(t_1)e^{\mu t_1}}{e^{\mu t}} + \frac{c \norm{x^*}^2 }{2} \dfrac{\dot{\b}(t)}{\b(t)}+\frac{c}{2}\|x^*\|^{2}\right).
\end{eqnarray*}
We conclude, for $t$ large enough, the estimation \eqref{eq:8-3}. 
\end{proof}

\section{Particular cases on the choice of $\b(t)$}
We illustrate the theoretical conditions on $\beta(t)$ by two interesting examples:
\subsection{Case $\beta (t)=t^m\ln^pt$}
Consider the positive function $\beta (t)=t^m\ln^pt$ for $t\geq t_0>0$ and $(m,p)\in \mathbb R_+^2\setminus\{(0,0)\}$. We have
\begin{eqnarray*}
&&\dot\b(t) = t^{m-1}\ln^{p-1}t \left(m\ln t +p\right),\\
&&\ddot\b(t) = t^{m-2}\ln^{p-2}t \left(m(m-1)\ln^2 t + (2m-1)p\ln t + p(p-1)\right).\\
\end{eqnarray*}
Then, for $t$ large enough, we have the function $\frac{\dot\b(t)}{ \b(t)}$ (respectively $\frac{\ddot\b(t)}{\dot \b(t)}$) is equivalent to $\frac{p}{t\ln t}$ (respectively to $\frac{m-1}{t}$ if $m\neq 1$, and $\frac{1}{t\ln t}$ if $m= 1$).\\
 Thus
	$
	\underset{t\rightarrow +\infty}{\limsup} \,\frac{  1 + \frac{\dot{\beta } (t)}{\beta  (t)}}{\mu + \frac{\ddot{\beta } (t)}{{\dot\beta} (t)} - \frac{\dot{\b}(t)}{\b(t)}}= \dfrac1{\mu} <+\infty,
	$ for each $(m,p)\in \mathbb R_+^2\setminus\{(0,0)\}$.\\
	We conclude that condition \ref{H0} is satisfied whenever $c>0$ and $0<\m\leq c-1$.
	
	\begin{corollary}\label{CorCont1}
	If $f: \mathcal{H} \rightarrow \mathbb{R}$  satisfies  \ref{H01}, $\beta(t)=t^m\ln^pt$ for $t\geq t_0>0, (m,p)\in \mathbb R_+^2\setminus\{(0,0)\}$ and $x(.): [t_0 , +\infty[ \rightarrow \mathcal{H}  $ is a solution of  \eqref{systeme continue}. Then
 $x(t)$ strongly converges to $x^*$, and  for  $t \geq t_0$  large enough
\begin{eqnarray}
\label{eq:20-1}&&f(x(t)) - \min_{\mathcal H}f = \mathcal{O}\left( \dfrac{1}{t^m\ln^pt}\right);\\
\label{eq:20-2}&& \|x(t) - y(t)\|^2  = \mathcal{O}\left( \dfrac{1}{t\ln t}\right);\\
\label{eq:20-3}&&\norm{\nabla f(x(t))}^2 	=\left\{\begin{array}{ll} 
\mathcal{O}\left( \dfrac{1}{t\ln t}\right) & \text{ if either } m>\frac12 \text{ and } p>0,\\
& \text{ or }m=\frac12 \text{ and } p\geq1\\
\mathcal{O}\left( \dfrac{1}{t^{\min(1,2m)}\ln^p t}\right) & \text{ if either } m<\frac12 \text{ and } p>0, \\
& \text{ or }m=\frac12 \text{ and }0< p<1.
\end{array}\right.
\end{eqnarray}
	\end{corollary}

\subsection{Case $\beta (t)=t^me^{\gamma t^r}$}
If $\beta (t)=t^me^{\gamma t^r}$ for $t\geq t_0>0, \; m\geq 0,\; \g>0$ and $ 0<r\leq 1$, then we have
\begin{eqnarray*}
&&\dot\b(t) = t^{m-1}e^{\gamma t^r} \left(m +r\g t^r\right),\\
&&\ddot\b(t) =  t^{m-2}e^{\gamma t^r} \left(m(m-1) + (2m+r-1)r\g t^r + r^2\g^2 t^{2r} \right).\\
\end{eqnarray*}
Then, for $t$ large enough, we have the function $\frac{\dot\b(t)}{ \b(t)}$ (respectively $\frac{\ddot\b(t)}{\dot \b(t)}$) is equivalent to $\frac{r\g}{ t^{1-r}}$ (respectively to  $\frac{r\g}{ t^{1-r}}$).\\
 Thus
	$$
	\underset{t\rightarrow +\infty}{\limsup} \,\frac{  1 + \frac{\dot{\beta } (t)}{\beta  (t)}}{\mu + \frac{\ddot{\beta } (t)}{{\dot\beta} (t)} - \frac{\dot{\b}(t)}{\b(t)}}= \left\{\begin{array}{lll}
	\dfrac{1}{\mu} <+\infty &\text{ if }& r<1,\\
	\dfrac{1+\g}{\mu} <+\infty &\text{ if }& r=1.
	\end{array}\right.
	$$
	We conclude that conditions on $\b$ are satisfied whenever  $c>0$ and $0<\m\leq c-1$.
	\begin{corollary}\label{CorCont2}
	If $f: \mathcal{H} \rightarrow \mathbb{R}$  satisfies  \ref{H01}, $\beta (t)=t^me^{\gamma t^r}$ for $t\geq t_0>0, \; m\geq 0,\; \g>0,\; 0<r\leq 1$, and $x(.): [t_0 , +\infty[ \rightarrow \mathcal{H}  $ is a solution of  \eqref{systeme continue}. Then
 $x(t)$ strongly converges to $x^*$, and  for  $t \geq t_0$  large enough
\begin{eqnarray}
\label{eq:20-1}&&f(x(t)) - \min_{\mathcal H}f = \mathcal{O}\left( \dfrac{1}{t^me^{\gamma t^r}}\right);\\
\label{eq:20-2}&& \|x(t) - y(t)\|^2  = \mathcal{O}\left( \dfrac{1}{t^{1-r}}\right);\\
\label{eq:20-3}&&\norm{\nabla f(x(t))}^2 	=\mathcal{O}\left( \dfrac{1}{t^{1-r}}\right).
\end{eqnarray}
	\end{corollary}

\if{	
\subsection{Case $\beta (t)=e^{\m t} e^{\frac{\d}{\m}e^{\m t}}$}
If $\beta (t)=e^{\m t} e^{\frac{\d}{\m}e^{\m t}}$ for $t\geq t_0>0, \; \m,\d >0$, then we have
\begin{eqnarray*}
&&\dot\b(t) = \left(\m+ \d e^{\m t} \right) e^{\m t+\frac{\d}{\m}e^{\m t}},\\
&&\ddot\b(t) =  \left(\m^2+(3\m +\d )\d e^{\m t} \right) e^{\m t+\frac{\d}{\m}e^{\m t}}.\\
\end{eqnarray*}
Then, for $t$ large enough, we have the function $\frac{\dot\b(t)}{ \b(t)}$ (respectively $\frac{\ddot\b(t)}{\dot \b(t)}$) is equivalent to $\m+ \d e^{\m t} $ (respectively to  $3\m +\d  $).\\
 Thus
	$$
	\underset{t\rightarrow +\infty}{\limsup} \,\frac{  1 + \frac{\dot{\beta } (t)}{\beta  (t)}}{\mu + \frac{\ddot{\beta } (t)}{{\dot\beta} (t)} - \frac{\dot{\b}(t)}{\b(t)}}= \frac{\m+\d}{2\m} <+\infty	.
	$$
	We conclude that conditions on $\b$ are satisfied whenever  $c>0$ and $0<\m\leq c-1$.
	\begin{corollary}\label{CorCont2}
	If $f: \mathcal{H} \rightarrow \mathbb{R}$  satisfies  \ref{H01}, $\beta (t)=e^{\m t} e^{\frac{\d}{\m}e^{\m t}}$ for $t\geq t_0>0, \; \m,\d >0$, and $x(.): [t_0 , +\infty[ \rightarrow \mathcal{H}  $ is a solution of  \eqref{systeme continue}. Then
 $x(t)$ strongly converges to $x^*$, and  for  $t \geq t_0$  large enough
\begin{eqnarray}
\label{eq:20-1}&&f(x(t)) - \min_{\mathcal H}f = \mathcal{O}\left( \dfrac{1}{t^me^{\gamma t^r}}\right);\\
\label{eq:20-2}&& \|x(t) - y(t)\|^2  = \mathcal{O}\left( \dfrac{1}{t^{1-r}}\right);\\
\label{eq:20-3}&&\norm{\nabla f(x(t))}^2 	=\mathcal{O}\left( \dfrac{1}{t^{1-r}}\right).
\end{eqnarray}
	\end{corollary}
}\fi

\begin{figure} 
 \includegraphics[scale=0.35]{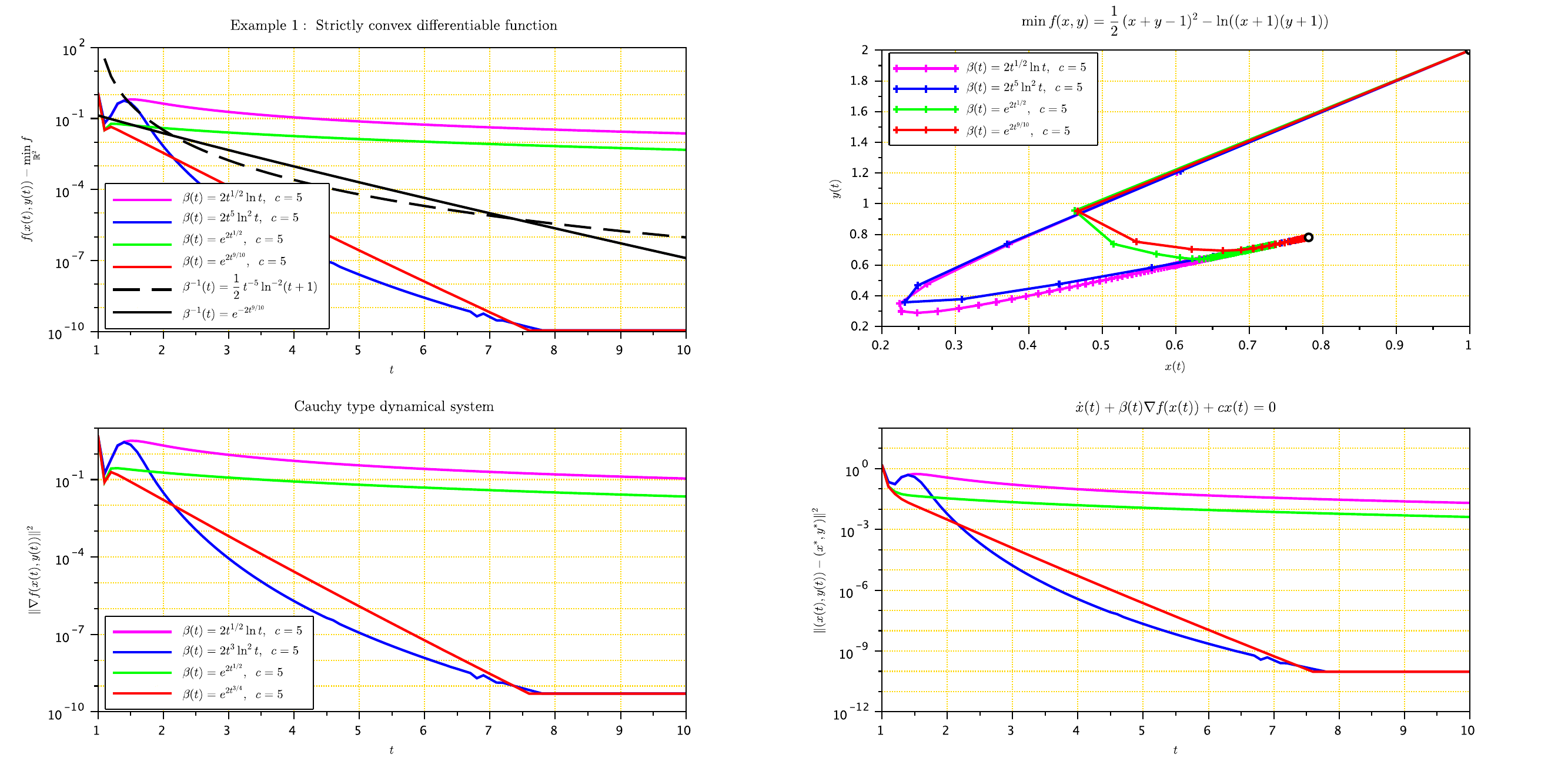}
  \caption{Convergence rates of
values, trajectories  and gradients.}
 \label{fig:trigs-c} 
\end{figure}
\if{
\section{Let's go from the first ordre to the implicit second order}
Having in view to develop rapid optimization methods, our study focuses on the second-order time evolution system obtained by applying the "time scaling and averaging" method of Attouch, Bot and Ngyuen \cite{ABotNguyen} to the dynamic \eqref{damped-id}.
In doing so, we hope to take advantage of both the fast convergence properties attached to the inertial gradient methods with vanishing damping coefficients and the strong convexity properties attached to Tikhonov's regularization.

\noindent So, let us make the change of time variable $s=\tau (t)$ in  the damped inertial dynamic 
\begin{equation}\label{damped-idz}
 \dot{z}(s) +  \b(t)\nabla f (z(s)) + cz(s) =0,
\end{equation}
where we take $z$ as a state variable and $s$ as a time variable, which will end up with $x$ and $t$ after time  scaling and averaging.  The time scale $\tau (\cdot)$ is a $\mathcal C^2$  increasing function from $[t_0,+\infty[$ to $[s_0,+\infty[$, which satisfies $\lim_{t \to +\infty}\tau (t) = + \infty$.  Setting $s=\tau (t)$ and  multiplying (\ref{damped-idz}) by $\dot\tau(t)>0$, we obtain
\begin{eqnarray}
\dot\tau(t)\dot{z}(\tau(t)) +  \dot\tau(t)\b(\tau(t))\nabla f (z(\tau(t))) + c\dot\tau(t) z(\tau(t))  &=& 0.  \label{change var1}
\end{eqnarray}
Set $v(t):= z(\tau(t))$.
By the derivation chain rule, we have
$
\dot{v} (t)= \dot{\tau}(t) \dot{z}(\tau(t))
$.
So reformulating  (\ref{change var1}) in terms of $v(\cdot)$ and its derivatives, we obtain
\begin{equation}\label{1st-damped-id}
\dot{v} (t) +  \dot\tau(t)\b(\tau(t))\nabla f (v(t)) + c\dot\tau(t)  v(t)=0.
\end{equation}
Let us now consider the averaging process which consists in passing from $v$ to $x$ which is defined by means of the first-order linear differential equation
\begin{equation}\label{1st-damped-idd}
v(t)=x(t) + \dot\tau(t)\dot x(t). 
\end{equation}
By temporal derivation of \eqref{1st-damped-idd} we get
\begin{equation}\label{1st-damped-iddd} 
\dot v(t) = \dot x(t) + \ddot\tau(t)\dot x(t) + \dot\tau(t)\ddot x(t).
\end{equation}
Replacing $v(t)$ and  $\dot v(t)$ in \eqref{1st-damped-id}, we get
\begin{equation}\label{1st-damped-id4}
\dot\tau(t) \ddot x(t) +(1+ \ddot\tau(t))\dot x(t) + \dot\tau(t)\b(\tau(t))\nabla f\left[x(t)+\dot\tau(t)\dot x(t)\right] + c\dot\tau(t)  \left[x(t)+\dot\tau(t)\dot x(t)\right]=0.
\end{equation}
Dividing by $\dot\tau(t) >0$, we finally obtain
\begin{equation}\label{1st-damped-id505}
 \ddot x(t) +\frac{1+ \ddot\tau(t)}{\dot\tau(t)}\dot x(t) + \b(\tau(t))\nabla f\left[x(t)+\dot\tau(t)\dot x(t)\right] + c \left[x(t)+\dot\tau(t)\dot x(t)\right]=0.
\end{equation}
According to the Su, Boyd and Cand\`es \cite{SBC} model for the Nesterov accelerated gradient method, we consider the case where the viscous damping coefficient in \eqref{1st-damped-id505} satisfies $ \frac{1+ \ddot\tau(t)}{\dot\tau(t)} = \frac{\alpha}{t}$ for some $\alpha>1$. Thus $\tau(t) = \dfrac{t^2}{2(\alpha - 1)}$, and the system \eqref{1st-damped-id505} becomes
\begin{equation}\label{1st-damped-id55}
 \ddot x(t) +\frac{\alpha}{t}  \dot x(t) + \b(\tau(t))\nabla f\left[x(t)+\dfrac{t}{\alpha - 1}\dot x(t)\right] + c \left[x(t)+\dfrac{t}{\alpha - 1}\dot x(t)\right]=0.
\end{equation}

%
%
\begin{theorem}\label{thm:model-inertial-b}
	Take    $\alpha >3$ and $\delta >1$, which gives $\alpha + 2\delta >5$.
	Let $x : [t_0, +\infty[ \to \mathcal{H}$ be a solution trajectory of
	\eqref{2d-damped-id-p111}.		
Then,  the following properties are satisfied. 	
	\begin{eqnarray}
	&&(i) \,\,  f(x(t))-\min_{\cH} f= \mathcal O\left( \frac{1}{t^{2}}  \right) \mbox{ as } \; t \to +\infty;\label{contr:fx(t)22}\\ 
	&& (ii) \,\, \mbox{ There is   strong convergence of } x(t)
	 \mbox{ to the minimum norm solution } x^*.\nonumber\\
	 &&(iii) \,\,
\| \nabla f(x(t)) \| = \mathcal O\left( \dfrac{1}{t} \right)  \mbox{ as } \; t \to +\infty.
	\end{eqnarray}
\end{theorem}

\begin{proof}
According to Theorem \ref{{Th1}, the rescaled function $v(t):= z(\tau(t))$ satisfies
	\begin{eqnarray}
	&&
	f(v(t))-\min_{\cH} f= \mathcal O\left( \frac{1}{\b(t)}  \right)  \mbox{ as } \; t \to +\infty;\label{contr:fv(t)20}
	\end{eqnarray}
	Our objective  is now to obtain the corresponding convergence rate of the values $f(x(t))-\min_{\cH} f$ as $t \to +\infty$, and the strong convergence of 
 $x(t)$ to $x^*$   the minimum norm element of $S$.
\if{
 The following proof is inspired by  Attouch, Bot, Ngyuen \cite{ABotNguyen}. It highlights the averaging interpretation of the passage from $v$ to $x$.
\\
$(i)$  Let us rewrite \eqref{1st-damped-idd}  as
\begin{equation}\label{change var28}
	t \dot{x}(t) + (\alpha -1) x(t) = (\alpha -1) v(t).
\end{equation}
\noindent After multiplication of  \eqref{change var28} by $t^{\alpha -2}$, we get equivalently
\begin{equation}\label{change var29}
	t^{\alpha -1} \dot{x}(t) + (\alpha -1)t^{\alpha -2} x(t) = (\alpha -1)t^{\alpha -2} v(t),
\end{equation}
that is
\begin{equation}\label{change var30}
	\frac{d}{dt} \left( t^{\alpha -1}x(t)\right)  = (\alpha -1)t^{\alpha -2} v(t).
\end{equation}
By integrating \eqref{change var30} from $t_{0}$ to $t$, and according to  $ x(t_{0})=x_0$, we obtain
\begin{eqnarray}
	x(t) &=&  \frac{t_{0}^{\alpha -1}}{ t^{\alpha -1}} v(t_0) + \frac{\alpha -1}{t^{\alpha -1}}\int_{t_{0}}^t \theta^{\alpha -2} v(\theta)d\theta \label{def:x}
\end{eqnarray}
where for simplicity we take $v(s_0)=x_0$.  
Then, observe that $x(t)$ can be simply written as follows
\begin{equation}\label{proba-formulation}
	x(t) =   \int_{t_{0}}^t v(\theta)\,  d\mu_{t} (\theta) 
\end{equation}
where $\mu_t$ is the positive  Radon  measure on $[t_{0}, t]$ defined by
$$
\mu_t = \frac{t_{0}^{\alpha -1}}{ t^{\alpha -1}} \delta_{t_{0}} +  (\alpha -1) \frac{\theta^{\alpha -2}}{t^{\alpha -1}} d\theta .
$$
We are therefore led to examining the convergence rate of  $f\left(\int_{t_{0}}^t v(\theta)  d\mu_{t} (\theta)\right) -\min_{\cH} f$ towards zero as $t \to + \infty$.
We have  that $\mu_t$ is a positive Radon measure on $[t_{0}, t]$  whose total mass is equal to $1$. It is therefore a probability measure, and $\int_{t_{0}}^t v(\theta)\,  d\mu_{t} (\theta)$ is obtained by \textbf{averaging} the trajectory $v(\cdot)$ on $[t_{0},t]$ with respect to  $\mu_t$.
From there, we can deduce fast convergence properties of $x(\cdot)$.
According to the convexity of $f$, and  Jensen's inequality, we obtain that
\begin{eqnarray*}
	f\left(\int_{t_{0}}^t v(\theta)\,  d\mu_{t} (\theta)\right)  -\min_{\cH} f &=& (f -\min_{\cH} f ) \left( \int_{t_{0}}^t v(\theta)  d\mu_t (\theta)\right)\\
	&\leq& \int_{t_{0}}^t  \left( f (v(\theta)) -\min_{\cH} f \right) d\mu_t (\theta)
	\\
	&\leq& L_f \int_{t_{0}}^t  \frac{1}{\theta^2} d\mu_t (\theta) \quad \forall t \geq t_0,
\end{eqnarray*}
where the last  inequality above comes from \eqref{contr:fv(t)20}.
According to the definition of $\mu_t$, it yields
\begin{eqnarray*}
	f\left( \int_{t_{0}}^t v(\theta)\,  d\mu_{t} (\theta)\right)  -\min_{\cH} f 
	&\leq&  \frac{L_f t_{0}^{\alpha -3}}{ t^{\alpha -1}}  +    L_f(\alpha -1) \frac{1}{t^{\alpha -1}}  \int_{t_{0}}^t  \theta^{\alpha -4} d\theta\\
	&\leq&\frac{L_f t_{0}^{\alpha -3}}{ t^{\alpha -1}} + L_f \frac{ \alpha-1}{\alpha-3}\left( \frac{1 }{ t^2} - \frac{t_{0}^{\alpha -3}}{t^{\alpha -1}}\right)
	 \quad \forall t \geq t_0.
\end{eqnarray*}
According to  \eqref{proba-formulation} we get 
\medskip
\noindent $\bullet$ For $1 <\alpha <3$, 
$$f(x(t)) -\min_{\cH} f \leq \frac{C}{t^{\alpha -1}}.$$
$\bullet$ For $\alpha >3$, we have $\dfrac{t_0^{\alpha-3}}{ t^{\alpha -1}} \leq \dfrac{1 }{ t^2}$  for every $t \geq t_0$. We therefore obtain 
$$
 f(x(t))-\min_{\cH} f= \mathcal O\left( \frac{1}{t^{2}}  \right) \mbox{ as } \; t \to +\infty
 $$
$(ii)$ According to Theorem \ref{thm:model-ab}, the rescaled function $v(\cdot )$  strongly converges to the minimum norm solution. From  the interpretation of $x$ as an average of $v$, and  using the fact that convergence implies ergodic convergence, we deduce the strong convergence of $x(t)$ towards the minimum norm solution.
This argument is detailed in the next paragraph, so we omit the details here.
\medskip
$(iii)$ We have obtained in $ii)$ that the trajectory $x(\cdot)$ converges. Therefore, it is bounded. Let $L>0$ be the Lipschitz constant of   $\nabla f$ on a ball that contains the trajectory  $x(\cdot)$.   According to the convergence rate of values  and Lemma \ref{ext_descent_lemma} in the Appendix, we immediately obtain the following  convergence rate of the gradients towards zero
\vspace{-7pt}
\begin{equation*}
	\dfrac{1}{2L} \left\lVert \nabla f(x(t)) \right\rVert ^{2} \leq f(x(t)) -\min_{\cH} f =\mathcal O\left( \frac{1}{t^2} \right).
\end{equation*}
So,
$$ \|\nabla f (x(t))\|=\mathcal{O}\left(\dfrac{1}{t} \right) \mbox{ as } \; t \to +\infty.
$$
This completes the proof. \qed
\end{proof}
\subsection{$\varepsilon (t)=\frac1{t^r}$ for $0<r<1$ }
The system \eqref{1st-damped-id55} becomes
\begin{equation}\label{2d-damped-id-p1}
 \ddot x(t) +  \frac{\alpha}{t}\dot x(t) + \nabla f\left[x(t)+\dfrac{s}{\alpha - 1}\dot x(t)\right] + \left(\dfrac{2(\alpha - 1)}{t^2}\right)^{r} \left[x(t)+\dfrac{t}{\alpha - 1}\dot x(t)\right]=0.
\end{equation}
}\fi
According to Theorem \ref{thm:model-a}, the rescaled function $v(t):= z(\tau(t))$ satisfies
	\begin{eqnarray}
	&&
	f(v(t))-\min_{\cH} f= \mathcal O \left( \displaystyle\frac{1}{\tau(t)^{r} }   \right) \mbox{ as } \; t \to +\infty;\label{contr:fv(t)2}\\
	&& \|v(t) -x_{\varepsilon(\tau(t))}\|^2=\mathcal{O}\left(\dfrac{1}{ \tau(t)^{1-r}}\right) \mbox{ as } \; t \to +\infty. \label{contr:v(t)2}
	\end{eqnarray}
	Our objective  is now to obtain the corresponding convergence rate of the values $f(x(t))-\min_{\cH} f$ as $t \to +\infty$, and strong convergence of 
 $x(t)$ to $x^*$   the minimum norm element of $S$.

\begin{theorem}\label{thm:model-inertial}
	Take   $0<r< 1$ and $\alpha >1$.
	Let $x : [t_0, +\infty[ \to \mathcal{H}$ be a solution trajectory of
\begin{equation}\label{2d-damped-id-p1-t}
\ddot x(t) + \left( \frac{\alpha}{t}  +\dfrac{2^r (\alpha - 1)^{r-1}}{t^{2r-1}}  \right)  \dot x(t) + \nabla f\left(x(t)+\dfrac{t}{\alpha - 1}\dot x(t)\right) + 
\dfrac{2^r(\alpha - 1)^r}{t^{2r}} x(t)=0.
\end{equation}		
Then,  the following properties are satisfied. 	
	\begin{eqnarray}
	&&(i) \,\,  f(x(t))-\min_{\cH} f= \mathcal O\left( \frac{1}{t^{\alpha - 1}} + \frac{1}{t^{2r}} \right) \mbox{ as } \; t \to +\infty;\label{contr:fx(t)2}\\
	&& (ii)  \, \,\mbox{ There is   strong convergence of } x(t) \mbox{ to the minimum norm solution } x^* ; \\
	&& (iii) \,\, \| \nabla f(x(t)) \|^2 = \mathcal O\left( \frac{1}{t^{\alpha - 1}} + \frac{1}{t^{2r}} \right) \mbox{ as } \; t \to +\infty. \label{contr:x(t)2}
	\end{eqnarray}
\end{theorem}
\begin{proof}
According to Theorem \ref{Th1}, the rescaled function $v(t):= z(\tau(t))$ satisfies
	\begin{eqnarray}
	&&
	f(v(t))-\min_{\cH} f=f(z(\tau(t)))-\min_{\cH} f = \mathcal O\left( \frac{1}{\b(\tau(t))}  \right)  \mbox{ as } \; t \to +\infty;\label{contr:fv(t)20}
	\end{eqnarray}
	Our objective  is now to obtain the corresponding convergence rate of the values $f(x(t))-\min_{\cH} f$ as $t \to +\infty$, and the strong convergence of 
 $x(t)$ to $x^*$   the minimum norm element of $S$.

$(i)$ \textit{Convergence rates to zero for the values $f(x(t))-\min_\cH f$}.
By definition of $x(\cdot)$
$$
v(t)=x(t)+\dfrac{t}{\alpha - 1}\dot x(t) =\frac1{(\alpha - 1)t^{\alpha - 2}}\frac{d}{dt}(t^{\alpha - 1}x(t)) .
$$
Equivalently
\begin{equation}\label{eq:Eq-dif1}
\frac{d}{dt}(t^{\alpha - 1}x(t)) = ({\alpha - 1})t^{\alpha - 2}v(t).
\end{equation}
After integration from $t$ to $t+h$ of \eqref{eq:Eq-dif1}  and division by $(t+h)^{\alpha - 1}$, we get
\begin{eqnarray*}
x(t+h) &= & \left(\frac{t}{t+h}\right)^{\alpha - 1}x(t)+ \frac{1}{(t+h)^{\alpha - 1}}\int_t^{t+h}({\alpha - 1})s^{\alpha - 2}v(s)ds\\
	  &= & \left(\frac{t}{t+h}\right)^{\alpha - 1}x(t)+ \left(1-\left(\frac{t}{t+h}\right)^{\alpha - 1}\right)\frac{1}{(t+h)^{\alpha - 1}-t^{\alpha - 1}} \int_t^{t+h}({\alpha - 1})s^{\alpha - 2}v(s)ds\\
	  &= & \left(\frac{t}{t+h}\right)^{\alpha - 1}x(t)+ \left(1-\left(\frac{t}{t+h}\right)^{\alpha - 1}\right)\frac{1}{(t+h)^{\alpha - 1}-t^{\alpha - 1}} \int_{t^{\alpha - 1}}^{(t+h)^{\alpha - 1}} v\left(\theta^{1/(\alpha - 1)}\right)d\theta,
\end{eqnarray*}
where the last equality comes from the change of time variable $\theta=s^{\alpha - 1}$.

\noindent According to the convexity of the   function $F=f - \inf_{\cH}f$, and using  Jensen's inequality\footnote{Jensen's inequality means:  If $\phi : [a, b]\rightarrow {\cH}$ is an integrable
function and $F :  {\cH}\rightarrow  \R$ is a continuous convex function, then
$F\left(\frac1{b-a}\int_a^b\phi(t)dt\right) \leq \frac1{b-a}\int_a^bF(\phi(t))dt$}, we obtain
\begin{small}
\begin{eqnarray*}
F(x(t+h)) & \leq &  \left(\frac{t}{t+h}\right)^{\alpha - 1}F(x(t))+ \left(1-\left(\frac{t}{t+h}\right)^{\alpha - 1}\right)F\left(\frac{1}{(t+h)^{\alpha - 1}-t^{\alpha - 1}} \int_{t^{\alpha - 1}}^{(t+h)^{\alpha - 1}} v\left(\theta^{1/(\alpha - 1)}\right)d\theta \right)\\
	& \leq &  \left(\frac{t}{t+h}\right)^{\alpha - 1}F(x(t))+ \left(1-\left(\frac{t}{t+h}\right)^{\alpha - 1}\right)\frac{1}{(t+h)^{\alpha - 1}-t^{\alpha - 1}}  \int_{t^{\alpha - 1}}^{(t+h)^{\alpha - 1}} F\left( v\left(\theta^{1/(\alpha - 1)}\right)\right)d\theta  .
\end{eqnarray*}
\end{small}
Using again the change of time variable $s=\theta^{1/(\alpha - 1)}$, we get
$$
\int_{t^{\alpha - 1}}^{(t+h)^{\alpha - 1}}  F\left( v\left(\theta^{1/(\alpha - 1)}\right)\right)d\theta   =  \int_t^{t+h}({\alpha - 1})s^{\alpha - 2}F(v(s))ds .
$$
It follows that
\begin{small}
\begin{eqnarray*}
F(x(t+h))  &\leq&  \left(\frac{t}{t+h}\right)^{\alpha - 1}F(x(t))+ \left(1-\left(\frac{t}{t+h}\right)^{\alpha - 1}\right)\frac{1}{(t+h)^{\alpha - 1}-t^{\alpha - 1}}  \int_{t}^{t+h} ({\alpha - 1})s^{\alpha - 2}F(v(s))ds .
\end{eqnarray*}
\end{small}
Using  \eqref{contr:fv(t)20}, with  $\tau(t) = \dfrac{t^2}{2(\alpha - 1)}$, we obtain the existence of a constant $C>0$ and $t_1>t_0$ such that for $t\geq t_1$
$$
F(v(t)) = f(v(t))-\min_{\cH} f\leq  \frac{C}{\b(\tau(t))}   .
$$
Thus, for $t\geq t_1$
\begin{small}
\begin{eqnarray*}
F(x(t+h)) &\leq&   \left(\frac{t}{t+h}\right)^{\alpha - 1}F(x(t))+ \left(1-\left(\frac{t}{t+h}\right)^{\alpha - 1}\right)\frac{C(\alpha - 1)}{(t+h)^{\alpha - 1}-t^{\alpha - 1}}\int_t^{t+h} \displaystyle \frac{s^{\alpha - 2}}{\b(\tau(s))}  ds  . 
\end{eqnarray*}
\end{small}
Set
$$
 \Gamma (t):=t^{\alpha - 1}F(x(t)) - 2^rC(\alpha - 1)^{r+1}\int_{s_0}^{t}  \displaystyle \frac{s^{\alpha - 2}}{\b(\tau(s))}  ds ,
$$
then, for $t\geq t_1$
$$
\frac1h\left(  \Gamma (t+h)- \Gamma (t)\right)\leq 0
$$
By letting $h \rightarrow 0^+$ in the above inequality, we fget  that $\frac{d}{dt}\Gamma (t) \leq 0,$ which  implies $\Gamma (t)$ is nonincreasing. Therefore  for each $0<r<1$ and each $t\geq t_1$, we have $\Gamma (t) \leq \Gamma (t_1)$, which gives
$$
t^{\alpha - 1}F(x(t))  \leq \Gamma (t_1) + 2^rC(\alpha - 1)^{r+1}\int_{t_1}^{t} \displaystyle \frac{s^{\alpha - 2}}{\b(\tau(s))}  ds.
$$ 
Consequently,
\begin{equation}\label{ineq-estim}
f (x(t))- \inf_{\cH}f=F(x(t))\leq  \frac{\Gamma (t_1)}{t^{\alpha - 1}} +  \frac{2^rC(\alpha - 1)^{r+1}}{t^{\alpha - 1}}\int_{t_1}^{t} \displaystyle \frac{s^{\alpha - 2}}{\b(\tau(s))}  ds.
\end{equation}
We conclude that for $t$ large enough
$$
f (x(t))- \inf_{\cH}f= \mathcal O\left( \frac{1}{t^{\alpha - 1}}\left(1+ \displaystyle\int_{t_1}^{t}  \frac{s^{\alpha - 2}}{\b(\tau(s))}  ds \right) \right).
$$ 
As a consequence
\begin{itemize}
\item When $\alpha \geq 3$ we have
$$
f (x(t))- \inf_{\cH}f= \mathcal O\left(  \frac{1}{t^{2r}} \right)
$$
So we can get as close as possible to the optimal convergence rate $1/t^2$.

\item When $\alpha < 3$, by taking $ \frac{\alpha -1}{2}  <r<1$, we get
$$
f (x(t))- \inf_{\cH}f= \mathcal O\left(  \frac{1}{t^{\alpha - 1}} \right).
$$
\end{itemize}
%

\noindent $(ii)$ \textit{ Strong convergence of $x(t)$ to the minimum norm solution}.

Let us go back to \eqref{eq:Eq-dif1}. We have
$$
\frac{d}{ds}(s^{\alpha - 1}x(s)) = ({\alpha - 1})s^{\alpha - 2}v(s).
$$
Integrating from $s_0$ to $t$, we get
\begin{eqnarray*}
x(t) &=& \left(\dfrac{s_0}{t}\right)^{\alpha - 1}x(s_0) + \frac{{\alpha - 1}}{t^{\alpha - 1}}\int_{s_0}^t \theta^{\alpha - 2}v(\theta)d\theta \\
&=&  \left(\dfrac{s_0}{t}\right)^{\alpha - 1}x(s_0) + \frac{{\alpha - 1}}{t^{\alpha - 1}}\int_{s_0}^t \theta^{\alpha - 2}(v(\theta)-x^*)d\theta + \left(\frac{{\alpha - 1}}{t^{\alpha - 1}}\int_{s_0}^t \theta^{\alpha - 2}d\theta\right) x^*\\
&=&  \left(\dfrac{s_0}{t}\right)^{\alpha - 1}x(s_0) + \frac{{\alpha - 1}}{t^{\alpha - 1}}\int_{s_0}^t \theta^{\alpha - 2}(v(\theta)-x^*)d\theta + \left(1-\left(\frac{s_0}{t}\right)^{\alpha - 1}\right) x^*\\
\end{eqnarray*}
Therefore,
\begin{eqnarray}\label{eq_cont-us2}
\|x(t) - x^*\| &\leq&   \left(\dfrac{s_0}{t}\right)^{\alpha - 1}|x(s_0) - x^*\| + \frac{{\alpha - 1}}{t^{\alpha - 1}}\int_{s_0}^t \theta^{\alpha - 2}\|v(\theta)-x^*\|d\theta.
\end{eqnarray}
By Theorem \ref{thm:model-a}, we have, as $\theta$ goes to $+\infty$,    strong convergence of $v(\theta)=z(\tau(\theta))$ to the minimum norm solution $x^*$. \\
Since $\lim_{s\rightarrow+\infty}\|v(\theta)-x^*\|=0$, then given $a > 0$, there exists $s_a>0$ sufficiently large so that $\|v(\theta)-x^*\|<a$, for $\theta \geq s_a$.

Then, for $t> s_a$, split the integral $\int_{s_0}^t \theta^{\alpha - 2}\|v(\theta)-x^*\|d\theta$ into two parts to obtain
\begin{eqnarray*}
\frac{\alpha - 1}{t^{\alpha - 1}}\int_{s_0}^t \theta^{\alpha - 2}\|v(\theta)-x^*\|d\theta &=& \frac{\alpha - 1}{t^{\alpha - 1}}\int_{s_0}^{s_a} \theta^{\alpha - 2}\|v(\theta)-x^*\|d\theta + \frac{\alpha - 1}{s^{\alpha - 1}}\int_{s_a}^t \theta^{\alpha - 2}\|v(\theta)-x^*\|d\theta\\
&\leq& \frac{\alpha - 1}{t^{\alpha - 1}}\int_{s_0}^{s_a} \theta^{\alpha - 2}\|v(\theta)-x^*\|d\theta + \frac{a(\alpha - 1)}{t^{\alpha - 1}}\int_{s_a}^t \theta^{\alpha - 2}d\theta\\
&=& a + \frac{1}{t^{\alpha - 1}} \left((\alpha - 1)\int_{s_0}^{s_a} \theta^{\alpha - 2}\|v(\theta)-x^*\|d\theta - as_a^{\alpha - 1}\right).
\end{eqnarray*}
Now let $t\rightarrow+\infty$ to deduce that
$$
\limsup_{t\rightarrow+\infty}\frac{\alpha - 1}{t^{\alpha - 1}}\int_{s_0}^t \theta^{\alpha - 2}\|v(\theta)-x^*\|d\theta \leq a.
$$
Since this is true for any $a > 0$, this insures 
$$
\lim_{t\rightarrow+\infty}\frac{\alpha - 1}{t^{\alpha - 1}}\int_{s_0}^t \theta^{\alpha - 2}\|v(\theta)-x^*\|d\theta=0.
$$
Going back to \eqref{eq_cont-us2}, we conclude for $\alpha - 1>0$ that $\lim_{t\rightarrow+\infty}\|x(t) - x^*\|=0$. This means that $x(t)$ strongly  converges to $x^*$ as $t\rightarrow +\infty$.

\medskip

$iii)$\textit{ Convergence of the gradients towards zero.} We have obtained in $ii)$ that the trajectory $x(\cdot)$ converges. Therefore, it is bounded. Let $L>0$ be the Lipschitz constant of   $\nabla f$ on a ball that contains the trajectory  $x(\cdot)$.   According to the convergence rate of values  and Lemma \ref{ext_descent_lemma} in the Appendix, we immediately obtain the following  convergence rate of the gradients towards zero
\vspace{-7pt}
\begin{equation*}
	\dfrac{1}{2L} \left\lVert \nabla f(x(t)) \right\rVert ^{2} \leq f(x(t)) -\min_{\cH} f =\mathcal O\left( \frac{1}{t^{\alpha - 1}} + \frac{1}{t^{2r}} \right).
\end{equation*}	
This completes the proof.
\end{proof}

}\fi

\section{Example for comparison of the convergence rate}
\begin{example}\label{exemple1} 
Take $f  :  ]-1,+\infty[^{2} \to \R$ which is defined by 
$$
f(x)=\frac12(x_1+x_2)^2-\ln(x_1+1)(x_2+1).
$$
 The function $f_1$ is strictly convex with gradient
$$
\nabla f(x)=\begin{bmatrix}
x_1+x_2-\frac1{x_1+1}\\
x_1+x_2-\frac1{x_2+1}
\end{bmatrix} \text{ and Hessian } \nabla^2 f(x)=\begin{bmatrix}1+\frac1{(x_1+1)^2} & 1\\1 & 1+\frac1{(x_2+1)^2}\end{bmatrix} 
.
$$
The unique minimizer of $f$ is $x^*= \left((\sqrt {3}-1)/2,(\sqrt {3}-1)/2\right)$. The corresponding trajectories to the system \eqref{systeme continue} are depicted in Figure \ref{fig:trigs-c}. We note that the convergence rates of the values in this numerical test are consistent with those predicted in Corollaries \ref{CorCont1} and \ref{CorCont2}, while the convergence rates of the gradients are clearly stronger than those predicted theoretically. 
Let us consider the times-second-order systems, treated in the very recent papers: 
\begin{eqnarray*}
\text{(TRAL)}&&\dot{x}(t)+2t^2\ln^2t\nabla f\left(x(t)\right) + 5x(t)=0,\\
\text{(TRAE)}&&\dot{x}(t)+2t^2e^{2t^{9/10}}\nabla f\left(x(t)\right) + 5x(t)=0,\\
\text{(TRISAL)}&&\ddot{x}(t)+5\dot{x}(t)+2t^2\ln^2t\nabla f\left(x(t)\right) + 5x(t)=0,\text{ see \cite{BCR1}},\\
\text{(TRISAE)}&&\ddot{x}(t)+5\dot{x}(t)+2t^2e^{2t^{4/5}}\nabla f\left(x(t)\right) + 5x(t)=0,\text{ see \cite{BCR1}},\\
\text{(TRISG)}&&\ddot{x}(t)+5t^{-4/5}\dot{x}(t)+\nabla f\left(x(t)\right) +t^{-8/5}x(t)=0,\text{ see  \cite{ABCR-JDE}}, \\
\text{(TRISH)}&&\ddot{x}(t)+5t^{-4/5}\dot{x}(t)+\nabla f\left(x(t)\right) + 2\nabla^2 f\left(x(t)\right)\dot{x}(t) + t^{-8/5}x(t)=0, \text{ see  \cite{ABCR-JDE}}.
\end{eqnarray*}
By comparing  the two times-first-order systems, \eqref{systeme continue} where $\beta (t)$ is either equal to $2t^2\ln^2t$ or $2t^2e^{2t^{9/10}}$, with those of second order, we are surprised (see Figure \ref{fig:trigs-d}) by the better rate of convergence brought by these two reduced and inexpensive systems.
\begin{figure} 
 \includegraphics[scale=0.35]{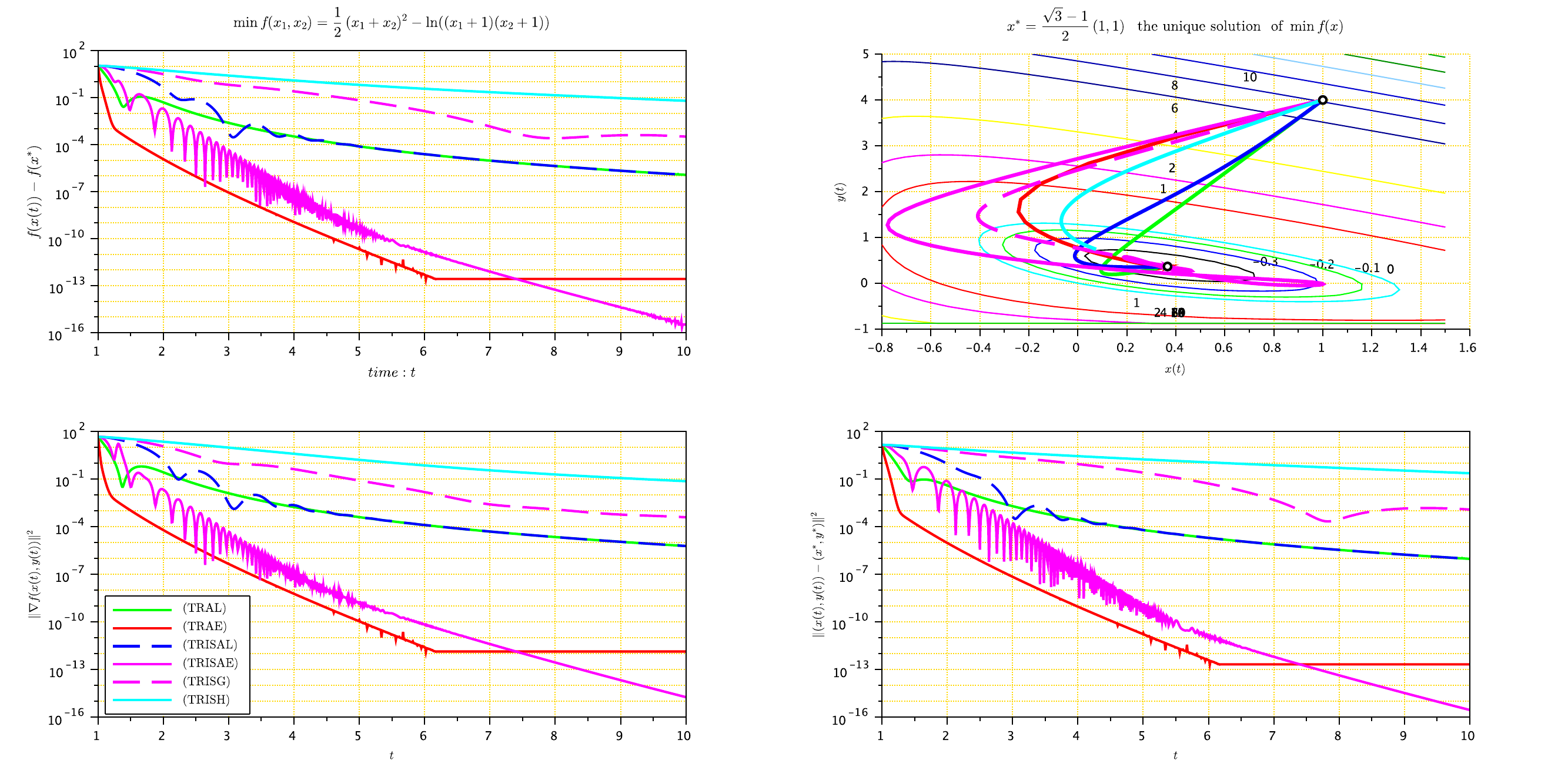}
  \caption{Convergence rates of
values, trajectories  and gradients.}
 \label{fig:trigs-d} 
\end{figure}
\end{example}

%

\section{Implicit discretization for nonsmooth convex functions}

Here, we suppose $f:\,\cH\rightarrow \mathbb R\cup\{ +\infty\}$ to  be a  proper lower semicontinuous convex function. Let us recall that an implicit discretization of the dynamical system \eqref{systeme continue} with $\beta(k)=\dfrac{\b_k}{d}>0$ and $d=1-c>0$ leads to the numerical scheme  \eqref{proximal algorithm} that reads for every $k \geq k_0$
\begin{eqnarray}
&&x_{k+1}-x_k+\dfrac{\b_k}{d}  \p \varphi_k(x_{k+1}) = \frac1d\left( x_{k+1} +\b_k\p f(x_{k+1})- dx_k\right)\ni 0		\label{alg_xk}\\
&& \iff \;x_{k+1}= \text{prox}_{\frac{\b_k}{d} \varphi_k}( x_k) =\text{prox}_{\b_k f}(d x_k) ,\nonumber
\end{eqnarray}
where $ \varphi _k:=f+ \frac{1-d}{2\b_k}\norm{\cdot} ^2 $.
We suppose the following condition:
\begin{equation*}
\tag*{$(\mathbf{H}_{\beta_k})$} 
  (\b_k)\textit{ is a nondecreasing  and satisfies }\, \underset{k\rightarrow +\infty}{\lim}\, \b_k= + \infty .
\end{equation*}

For each $k\geq k_0$, we denote by $y_{k}$ the unique minimizer of the strongly convex function
$ \varphi _k $, which means by the first order convex optimality condition
\begin{equation}\label{nn}
\p f (y_{k})+\dfrac{1-d}{\b_k} y_{k}\ni 0\; \iff \; y_k=\text{prox}_{\frac{1-d}{\b_k} f}(0).
\end{equation}

Let us recall (see \cite[Prop 2.6]{Brezis} for general maximal monotone operators) that the Tikhonov approximation curve $ k \mapsto y_{k}$ satisfies
\begin{equation}\label{x^*}
	\forall k\geq k_0, \;\; \norm{y_{k}}\leq \norm {x^*}
\hbox{ and }	\underset{k\rightarrow +\infty}{\lim} \norm{y_{k}-x^*} = 0.
\end{equation}
Now, for $\l$ a positive constant, we introduce the following  discrete energy function
\begin{equation}\label{E_k}
E_k = \b _k \left( \varphi_k (x_k)- \varphi_k (y_{k})  \right)+ \frac{\l}{2}\norm {x_k-y_{k-1}}^2, \;\; \text{for all }\,k\geq k_0.
\end{equation}
Before giving  this section's principal theorem, we need two very important lemmas. The first lemma relates the asymptotic behavior of the sequence $E_k$ to the convergence rate of values and iterations, the second  lemma provides a few properties of the viscosity curve $( y_{k})_k.$
\begin{lemma}\label{Lemme 1 discr}
Let $(x_k)$ be the sequence generated by the algorithm  \eqref{proximal algorithm}. Then for any $k \geq k_0$ we have:  
 \begin{equation}\label{con- val disr}
f(x_k)-\underset{\mathcal{H}}{\min} \,f \leq \dfrac{ E_k}{\b_k} + \dfrac{1-d}{2\b_k} \norm{x^*}^2
\end{equation}
and
\begin{equation}\label{conv-tra discr}
\norm{x_k-y_{k}}^2 \leq \dfrac{2E_k }{1-d}.
\end{equation}
Therefore, $x_k$ converges strongly to $x^*$ as soon as $\underset{k\rightarrow +\infty}{\lim} E_k = 0$.
\end{lemma}
\begin{proof}
 Observing  the definition of $\varphi _k$, one has
\begin{eqnarray}
f(x_k)-\underset{\mathcal{H}}{\min} \,f &=& \varphi _k (x_k) - \dfrac{1-d}{2\b_k}\norm{x_k}^2 -\varphi _k (x^*) + \dfrac{1-d}{2\beta _k}\norm{x^*}^2 		\nonumber\\
&=&  \left[  \varphi _k (x_k)-  \varphi _k (x_{\b_k})  \right] +\left[ \underbrace{\varphi _k (y_{k})- \varphi _k (x^*)}_{ \leq 0}\right] 		\nonumber\\
&&+\dfrac{1-d}{2\beta _k}\norm{x^*}^2  \underbrace{-\dfrac{1-d}{2\beta _k}\norm{x_k}^2}_{\leq 0}  		\nonumber\\
&\leq &  \varphi _k (x_k)-  \varphi _k (y_{k}) +\dfrac{1-d}{2\beta _k}\norm{x^*}^2 . \label{eq21}
\end{eqnarray}
From definition of $E_k$, we obtain
\begin{equation}\label{3d}
 \varphi _k (x_k)-  \varphi _k (y_{k}) \leq \dfrac{ E_k}{\beta _k},
\end{equation}
which, combined with \eqref{eq21}, gives \eqref{con- val disr}.\\
 Using the strong convexity of $ \varphi _k$, and $y_{k}:=\underset{\mathcal{H}}{\argmin }\, \varphi _k$, we obtain
\begin{equation*}
  \varphi _k (x_k)-  \varphi _k (y_{k})\geq \dfrac{1-d}{2\beta _k} \norm{x_k -y_{k}}^2 .
  \end{equation*}
By combining the last inequality  with \eqref{3d}, we get
\begin{equation*}
 \dfrac{ E_k}{\beta _k} \geq  \dfrac{1-d}{2\beta _k} \norm{x_k-y_{k}}^2 ,
 \end{equation*}
which implies the  inequality \eqref{conv-tra discr}.
Recall that $\underset{k\rightarrow +\infty}{\lim} \norm{y_{k}-x^*} = 0$, so according to  $\underset{k\rightarrow +\infty}{\lim} E_k = 0$, we deduce that the sequence $(x_k )$ converges strongly to $x^*$. 
\end{proof}
\begin{lemma}\label{lemma 2 discr}
For every $k\geq k_0$, the following properties are satisfied:
\begin{itemize}
\item[i)] $\varphi _k(y_{k}) - \varphi _{k+1}(y_{k+1})\leq \dfrac{(1-d)}{2}\left( \dfrac{1}{\b_k} - \dfrac{1}{\b_{k+1}}\right)\norm {y_{k+1}}^2 ,$
\item[ii)] $\norm {y_{k+1}-y_{k}}^2 \leq \dfrac{\b_{k+1}-\b_k}{\b_{k+1}}\< y_{k+1} , y_{k+1}-y_{k} \>$,  and 
\item[iii)]  
$ \norm {y_{k+1}-y_{k}} \leq  \dfrac{\b_{k+1}-\b_k}{\b_{k+1}} \norm {y_{k+1}} .$
\end{itemize}
\end{lemma}
\begin{proof}
{\it i)} Recall that $y_{k}$ is the unique minimizer of the  function
$ \varphi _k=f+ \frac{1-d}{2\b_k}\norm{\cdot} ^2 $, then 
\begin{equation*}
\varphi _k (y_{k}) \leq \varphi _k (y_{k+1}) =f(y_{k+1})+\frac{1-d}{2\b_k}\norm{y_{k+1}}^2 ,  
\end{equation*}
and also 
\begin{equation*}
-\varphi _{k+1} (y_{k+1})=-f(y_{k+1})-\frac{1-d}{2\b_{k+1}}\norm{y_{k+1}}^2.
\end{equation*}
Summing these two inequalities, we get the first statement of this lemma.

{\it ii)} By \eqref{nn}, we have
\begin{equation*}
\frac{d-1}{\b_k} y_{k}\in \p f(y_{k})\;\;\;\; \text{and}\;\;\;\; \frac{d-1}{\b_{k+1}} y_{k+1}\in\p  f(y_{k+1}).
\end{equation*}
The monotonicity of $\p f$ gives
\begin{equation*}
\left\< \frac{1-d}{\b_k} y_{k}- \frac{1-d}{\b_{k+1}} y_{k+1} , y_{k+1}-y_{k} \right\> \geq 0 ,
\end{equation*}
it follows that
\begin{equation*}
\frac{1-d}{\b_k}\norm{y_{k+1}-y_k}^2\leq \left( \dfrac{1-d}{\b_k} - \dfrac{1-d}{\b_{k+1}}\right) \< y_{k+1} , y_{k+1}-y_{k} \>  ,
\end{equation*}
which gives the second statement of this lemma. The last statement follows from Cauchy-Schwarz inequality.
\end{proof}

By adding the following hypothesis on $(\b_k)$:
\begin{equation*}\label{H02}
\tag*{$(\mathbf{H}_{\b_k})$}\left \{
\begin{array}{ll}
(i) & \dot{\b}_k := \b_{k+1}-\b_k  \neq 0, \text{ for } k \text{ large enough};\\
(ii)  &\underset{k\rightarrow +\infty}{\lim}\, \dfrac{\dot{\b} _{k+1}}{\dot{\b}_k} =\underset{k\rightarrow +\infty}{\lim}\, \dfrac{\b _{k+1}}{\b_k} =\ell >0;
\end{array}
\right.
\end{equation*}
let us show the main result of this section.
\begin{theorem}\label{Th 2 disc}
Let   $f: \mathcal{H} \rightarrow \mathbb{R}\cup\{ +\infty\}$ be a  proper lower semicontinuous convex function, with $\text{argmin}_{\mathcal{H}}f\neq\emptyset$, and $(x_k)$ be a sequence generated by the algorithm \eqref{proximal algorithm}. Let $\r\in]0,1-d[$ and suppose the condition \ref{H02}. 
Then, we have
\begin{itemize}
\item[\textbf{i)}]  for $k$ large enough  
\begin{eqnarray*}
&&f(x_k)-\underset{\mathcal{H}}{\min}f = \mathcal{O} \left( \dfrac{1}{ \beta _k} \right), \;\norm{\dot x_k }^2   = \mathcal{O} \left(    \dfrac{\dot{\b}_k}{\b_k}+e^{-\r k} \right) \\
&& \text{ and }\norm{ x_k -  y_{k} }^2   = \mathcal{O} \left(    \dfrac{\dot{\b}_k}{\b_k}+e^{-\r k} \right) ;
\end{eqnarray*}
\item[\textbf{ii)}]   if $\ell=1$ in \ref{H02}, the sequence $(x_k)$ generated by the algorithm \eqref{proximal algorithm} converges strongly to $x^*$  the element of minimum norm of $\text{argmin}_{\mathcal{H}}f$;
\item[\textbf{iii)}]   for $k$ large enough
$$
d(0,\partial f(x_{k+1}))= \mathcal{O} \left(    \dfrac{1}{\b_k}\right), 
$$
and then if $f$ is differentiable on $\mathcal{H}$,   we have strong convergence of the gradients to zero with the rate
\begin{equation}\label{eq:grad35}
\norm{ \nabla f(x_k)}  =  \mathcal{O} \left(    \dfrac{1}{\b_k}\right).
\end{equation}
\end{itemize}
\end{theorem}
\begin{proof}
To simplify the writing of the proof, we use the following notations: $v_k:=x_k -y_{k-1}$ and
for any sequence $(u_k)$ in $\mathcal{H}$, we write 
$ \dot{u}_k:=u_{k+1}-u_k 
.$\\
From \eqref{E_k}, we have
\begin{equation*}
\begin{array}{lll}
\dot{E}_{k} &\leq& \b _{k+1 } \left( \varphi_{k+1 } (x_{k+1 })- \varphi_{k+1 } (y_{k+1 })  \right)- \b _k \left( \varphi_k (x_k)- \varphi_k (y_{k})  \right)\\\\
&\, &+ \frac{\l}{2}\left( \norm {v_{k+1}}^2-\norm {v_k}^2\right)\\\\
&=& \b_k \left[ \varphi_{k} (x_{k+1 })- \varphi_{k } (x_{k })  \right]+\dot{\b}_k \left[ \varphi_{k} (x_{k+1 })- \varphi_{k } (y_k)  \right]\\\\
&\;&+\b_{k+1} \left[ \varphi_{k+1} (x_{k+1})-\varphi_{k} (x_{k+1})+\varphi_k (y_{k}) -\varphi_{k+1} (y_{k+1}  ) \right]\\\\
&\,&+ \frac{\l}{2}\left( \norm {v_{k+1}}^2-\norm {v_k}^2\right).
\end{array}
\end{equation*}
Using the definition of $\varphi_k$ and that $(\b_k)$ is nondecreasing, we deduce that
\begin{equation}\label{5}
\varphi_{k+1} (x_{k+1 })- \varphi_{k } (x_{k+1 })=\dfrac{1-d}{2}\left( \dfrac{1}{\b_{k+1}}-\dfrac{1}{\b_{k}} \right)\norm{x_{k+1}}^2\leq 0 ; 
\end{equation}
so, according to Lemma \ref{lemma 2 discr} $i)$, we obtain
\begin{equation*}
\begin{array}{lll}
\dot{E}_{k} &\leq& \b_k \left[ \varphi_{k} (x_{k+1 })- \varphi_{k } (x_{k })  \right]+\dot{\b} _k\left[ \varphi_{k} (x_{k+1})-\varphi_{k} (y_k)   \right] \\\\
&\;& +\dfrac{(1-d)}{2}\dfrac{\dot{\b}_k}{\b_k} \norm{y_{k+1}}^2 +\dfrac{\l}{2}\left( \norm {v_{k+1}}^2-\norm {v_k}^2\right).
\end{array}
\end{equation*}
Noting that $ \frac{\l}{2}\left( \norm {v_{k+1 }}^2-\norm {v_k}^2\right) \leq \l \< \dot{v}_k , v_{k+1} \>$ then
\begin{equation}\label{6d}
\begin{array}{lll}
\dot{E}_{k} &\leq& \b_k \left[ \varphi_{k} (x_{k+1 })- \varphi_{k } (x_{k })  \right]+\dot{\b}_{k} \left[ \varphi_{k} (x_{k+1})-\varphi_{k} (y_k)   \right] \\\\
&\;& +\dfrac{(1-d)\dot{\b}_k}{2\b_k} \norm{y_{k+1}}^2 +\l \< \dot{v}_k , v_{k+1} \>.
\end{array}
\end{equation}
Returning to form \eqref{disc syst } of algorithm \eqref{proximal algorithm},  we obtain
\begin{equation*}
\begin{array}{lll}
\dot{v}_{k}&=&v_{k+1}-v_{k} =  x_{k+1}-y_{k}- x_k+y_{k-1} = \dot{x}_k - \dot{y}_{k-1}.
\end{array}
\end{equation*}
Thus
\begin{equation}\label{7d}
\l \< \dot{v}_{k} , v_{k+1}\> =\l  \<\dot x_{k}, x_{k+1}-y_{k} \>-\l \<\dot{y}_{k-1} , x_{k+1}-y_{k} \> .
\end{equation}
By strong convexity of $\varphi _k$ and $-\frac{d}{\b_k} \dot{x}_k  \in  \p \varphi_k (x_{k+1})$, we have
\begin{equation}\label{8d}
\left\langle\frac{d}{\b_k} \dot{x}_k, x_{k+1}-y_{k} \right\rangle \leq -\left(\varphi_k(x_{k+1})-\varphi _k (y_{k})\right)-\frac{1-d}{2\b_k}\norm{x_{k+1}-y_{k} }^2.
\end{equation}
Also, using Lemma \ref{lemma 2 discr} $ii)$  and inequality in  \eqref{x^*}, we have for any $b>0$
\begin{equation}\label{9d}
\begin{array}{lll}
-\l   \<\dot{y}_{k-1} , x_{k+1}-y_{k}\> &\leq& \dfrac{\l b}{2}\norm{\dot{y}_{k-1}}^2+\dfrac{\l}{2b}\norm{x_{k+1}-y_{k}}^2\\\\
&\leq& \dfrac{\l b}{2}\dfrac{\dot{\b}^2_{k-1}}{\b^2_{k}}\norm{x^*}^2+\dfrac{\l}{2b}\norm{x_{k+1}-y_{k}}^2.
\end{array}
\end{equation}
Combining \eqref{7d}, \eqref{8d} and \eqref{9d}, we deduce 
\begin{equation}\label{10d}
\begin{array}{lll}
\l \< \dot{v}_{k} , v_{k+1}\> &\leq&\dfrac{-\l \b_k}{d}\left(\varphi_k(x_{k+1})-\varphi _k (y_{k})\right) +\dfrac{\l b}{2}\dfrac{\dot{\b}^2_{k-1}}{\b^2_{k}}\norm{x^*}^2\\\\
& &+\dfrac{\l}{2}\left(\dfrac{1}{b}-\dfrac{1-d}{d} \right)\norm{x_{k+1}-y_{k}}^2.
\end{array}
\end{equation}
Using again the strong convexity of  $\varphi _k$, we have
\begin{equation*}
\varphi _k (x_{k})-\varphi_k(x_{k+1})\geq \left\<\frac{d}{\b_k} \dot{x}_k, x_{k+1}-x_{k} \right\>+ \dfrac{1-d}{2\b_k}\norm{x_{k+1}-x_{k} }^2 =  \dfrac{1+d}{2\b_k}\norm{\dot{x}_k}^2,
\end{equation*}
which implies
\begin{equation}\label{11d}
\begin{array}{lll}
\b_k\left(\varphi_k(x_{k+1})-\varphi _k (x_{k})\right)&\leq& -\dfrac{1+d}{2}\norm{\dot{x}_k}^2 \leq 0.
\end{array}
\end{equation}

Returning to  \eqref{6d},\eqref{10d},\eqref{11d}  and using $\norm{y_{k+1}}\leq\norm{x^*}$, we obtain
\begin{equation}\label{12d}
\begin{array}{lll}
\dot{E}_{k} &\leq&  \frac{\l}{2}\left(\frac{1}{b}-\frac{1-d}{d} \right)\norm{x_{k+1}-y_{k}}^2  +\frac{1}{2}\left( \l b \frac{\dot{\b}^2_{k-1}}{\b^2_{k}}+(1-d)\frac{\dot{\b}_{k}}{\b_k} \right) \norm{x^*}^2 \\\\
&\,&+\left(\b_{k+1}-(1+\frac{\l}{d})\b _k\right)\left[ \varphi_{k} (x_{k+1})-\varphi_{k} (y_k)   \right].
\end{array}
\end{equation}
Set $\m:=\frac{\r}{1-\r}$, then $\m \in \left]0,\frac{1-d}{d}\right[$, and according to the definition of $E_{k}$, we get
\begin{equation*}
\begin{array}{lll}
\m E_{k+1} &=&\m  \b_{k+1}\left( \varphi_{k+1}(x_{k+1})-\varphi_{k+1}(y_{k+1}) \right)+\dfrac{\m \l}{2}\norm{ x_{k+1}-y_{k}}^2 \\\\
&=&\m  \b_{k+1}\left( \varphi_{k}(x_{k+1})-\varphi_{k}(y_k) \right)+\dfrac{\m \l}{2}\norm{ x_{k+1}-y_{k}}^2\\\\
& & +\m  \b_{k+1}\left( \varphi_{k+1}(x_{k+1})-\varphi_{k}(x_{k+1})+\varphi_{k}(y_k)-\varphi_{k+1}(y_{k+1}) \right)
\end{array}
\end{equation*}
Using \eqref{5} and Lemma \ref{lemma 2 discr} $i)$, we obtain
\begin{equation}
\begin{array}{lll}\label{13d}
\m E_{k+1} &\leq&\m  \b_{k+1}\left( \varphi_{k}(x_{k+1})-\varphi_{k}(y_k) \right)+\dfrac{\m (1-d)\dot{\b}_k}{2\b_k}  \norm{x^*}^2 \\\\
&\,&+\dfrac{\m \l}{2}\norm{x_{k+1}-y_{k}}^2.
\end{array}
\end{equation}
Adding  \eqref{12d} and \eqref{13d}, we conclude
\begin{eqnarray}
\dot{E}_{k}+ \m E_{k+1}&\leq&  \dfrac{\l}{2}\left(\underbrace{\dfrac{1}{b}-\dfrac{1-d}{d} +\m}_{A} \right)\norm{x_{k+1}-y_{k}}^2 		\nonumber\\
&\,& +\dfrac{1}{2}\left( \l b \dfrac{\dot{\b}^2_{k-1}}{\b^2_{k}}+(1-d)(1+\m)\dfrac{\dot{\b}_{k}}{\b_k} \right) \norm{x^*}^2 		\nonumber\\
&\,&+\left(\underbrace{(1+\m)\b_{k+1}-\left(1+\frac{\l}{d}\right)\b _k}_{B}\right)\left( \varphi_{k} (x_{k+1})-\varphi_{k} (y_k)   \right).		\label{eq33}
\end{eqnarray}
Choosing $\l$ such that for all $k$: ${\l>d\left( (1+\m)\dfrac{\b _{k+1}}{\b_k}-1 \right)}$, which  is valid since {$\left(\frac{\b _{k+1}}{\b_k}\right)$ is bounded}, we get 
$$
B=\dfrac {\b_k}{d}\left(d\left( (1+\m)\dfrac{\b _{k+1}}{\b_k}-1 \right)-\l\right)\leq 0.
$$ 
Choosing also $b=\dfrac{d}{1-(1+\m)d}$, which is positive  since $\m<\frac{1-d}{d}$,  we cancel the term $A$ in inequality \eqref{eq33}. Consequently, for all $k\geq k_0$  
\begin{equation}\label{14d}
\dot{E}_{k} +\m E_{k+1}\leq \g _k \norm{x^*}^2
\end{equation}
with
\begin{equation*}
\g_k :=\dfrac{1}{2}\left( \dfrac{\l d}{1-(1+\m )d}  \dfrac{\dot{\b}^2_{k-1}}{\b^2_{k}}+(1-d)(1+\m)\dfrac{\dot{\b}_{k}}{\b_k} \right).
\end{equation*}
Inequality \eqref{14d} is equivalent to 
$$
E_{k+1}+\left(\dfrac{\m}{1+\m }-1 \right) E_{k}\leq \left(1-\dfrac{\m}{1+\m } \right) \g_k \norm{x^*}^2. 
$$
Multiplying the last inequality by $e^{\r (k+1)}$ and using $\r = \dfrac{\m}{1+\m}$,  we obtain
\begin{equation*}
e^{\r (k+1)}E_{k+1}+(\r-1 ) e^{\r (k+1)} E_{k}\leq (1-\r)e^{\r (k+1)}\g_{k} \norm{x^*}^2.
\end{equation*}
Hence, for all $k\geq k_0$
\begin{equation*}
e^{\r (k+1)}E_{k+1}-e^{\r k}E_{k} +(e^{-\r}+\r-1)e^{\r (k+1)} E_{k}\leq (1-\r)e^{\r (k+1)}\g_{k} \norm{x^*}^2.
\end{equation*}
Remarking the function $y\mapsto e^{-y}+y-1$ is nonnegative on $\R_+$, we conclude that, for all $k\geq k_0$,
\begin{equation}\label{15d}
e^{\r (k+1)}E_{k+1}-e^{\r k}E_{k} \leq me^{\r (k+1)}\dfrac{\dot{\b}_{k}}{\b_k} \left( r\left(\dfrac{\dot{\b}_{k-1}}{\dot{\b}_k}\right)^2  \dfrac{\dot{\b}_{k}}{\b_k} + 1\right),
\end{equation}
with $r:= \dfrac{\l d }{(1-\r-d)(1-d)}$ and  $m:=\frac12(1-\r)(1-d)(1+\m)\norm{x^*}^2$.\\
\if{
After summing the above inequalities between $k_0$ and $k > k_0$, and dividing by $e^{\r (k+1)}$ we get
\begin{equation}\label{estim E_kkkk}
E_{k+1} \leq \dfrac{e^{\r k_0}E_{k_0}}{e^{\r(k+1)}} +\dfrac{(1-d)(1-\r)^2\norm{x^*}^2}{2e^{\r (k+1)}} \left(\sum\limits_{j=k_0}^k e^{\r (j+1)}\left( r \left(\dfrac{\dot{\b}_{j-1}}{\b_{j}}\right)^2+\dfrac{\dot{\b}_{j}}{\b_j} \right)\right)
\end{equation}
with $r:= \dfrac{\l d }{(1-\r-d)(1-d)}$.\\
}\fi
Using  the assumption  \ref{H02}, we justify 
$$
\underset{k\to +\infty}{\lim}  \frac{\dot{\b}_k}{\b_k}=\underset{k\to +\infty}{\lim}\left( \frac{\b_{k+1}}{\b_k} -1\right)=\ell -1  \hbox{  and }\underset{k\to +\infty}{\lim} \frac{\dot{\b}_k}{\dot{\b}_{k+1}}  \frac{\b_{k+1}}{\b_{k}}=1 .
$$
Thus,   we have
\begin{eqnarray*}
\underset{k\to +\infty}{\lim} 
\dfrac{e^{\r (k+1)}\dfrac{\dot{\b}_{k}}{\b_k} \left( r\left(\dfrac{\dot{\b}_{k-1}}{\dot{\b}_k}\right)^2  \dfrac{\dot{\b}_{k}}{\b_k} + 1\right)}{\dfrac{\dot{\b}_{k+1}}{\b_{k+1}} e^{\r(k+2)}-\dfrac{\dot{\b}_k }{ \b_{k}}e^{\r (k+1)}} &=&
\underset{k\to +\infty}{\lim}
\dfrac{  r\left(\dfrac{\dot{\b}_{k-1}}{\dot{\b}_k}\right)^2  \dfrac{\dot{\b}_{k}}{\b_k} + 1}{\dfrac{\b_k}{\b_{k+1}}\dfrac{\dot{\b}_{k}}{\dot{\b}_{k+1}} e^{\r}-1} \\
	&=& \dfrac{  \dfrac{r(\ell -1)}{\ell^2}+1}{e^{\r}-1} < +\infty.
\end{eqnarray*}
%
Therefore, there exist $M>0$ and $k_1\geq k_0$ such that for $k\geq k_1$ 
\if{
\begin{equation*}
m\leq \dfrac{1}{\frac{\b_{k+1}}{\b_k}\frac{\dot{\b}_{k}}{\dot{\b}_{k+1}} \left( r\left(\frac{\dot{\b}_{k-1}}{\dot{\b}_k}\right)^2   \frac{\dot{\b}_k}{\b_k} + 1 \right)} \left( e^\r -\frac{\dot{\b}_k}{\dot{\b}_{k+1}}  \frac{\b_{k+1}}{\b_{k}} \right) .
\end{equation*}
Hence
\begin{equation*}
m\frac{\b_{j+1}}{\b_j}\frac{\dot{\b}_{j}}{\dot{\b}_{j+1}}  
\left( r\left(\frac{\dot{\b}_{j-1}}{\dot{\b}_j}\right)^2 \dfrac{\dot{\b}_{j}}{\b_j} + 1\right)\leq e^\r -\frac{\dot{\b}_j}{\dot{\b}_{j+1}}  \frac{\b_{j+1}}{\b_{j}}.
\end{equation*}
Multiplying by $\dfrac{\dot{\b}_{j+1}}{m\b_{j+1}}e^{\r (j+1)}$  we obtain 
}\fi
\begin{equation}\label{16d}
me^{\r (k+1)} \dfrac{\dot{\b}_{k}}{\b_k} \left( r\left(\frac{\dot{\b}_{k-1}}{\dot{\b}_k}\right)^2  \dfrac{\dot{\b}_{k}}{\b_k} + 1\right) \leq  M\left(\dfrac{\dot{\b}_{k+1}}{\b_{k+1}} e^{\r(k+2)}-\dfrac{\dot{\b}_k }{ \b_{k}}e^{\r (k+1)} \right).
\end{equation}
We conclude, for all $k\geq k_1$, 
\begin{equation}\label{15d2}
e^{\r (k+1)}E_{k+1}-e^{\r k}E_{k} \leq M\left(\dfrac{\dot{\b}_{k+1}}{\b_{k+1}} e^{\r(k+2)}-\dfrac{\dot{\b}_k }{ \b_{k}}e^{\r (k+1)} \right) .
\end{equation}
After summing the inequalities \eqref{15d2} between $k_1$ and $k > k_1$, and dividing by $e^{\r (k+1)}$ we get
\if{
\begin{equation}\label{estim E_kkkk}
E_{k+1} \leq \dfrac{e^{\r k_1}E_{k_1}}{e^{\r(k+1)}} +\dfrac{M}{e^{\r (k+1)}} \left(\sum\limits_{j=k_1}^k e^{\r (j+1)}\left( r \left(\dfrac{\dot{\b}_{j-1}}{\b_{j}}\right)^2+\dfrac{\dot{\b}_{j}}{\b_j} \right)\right)
\end{equation}
with $r:= \dfrac{\l d }{(1-\r-d)(1-d)}$.\\
Combining \eqref{16d} and \eqref{estim E_kkkk}, we deduce that for $k > k_1 $
}\fi
\begin{equation*}
\begin{array}{lll}
E_{k+1} &\leq& \dfrac{e^{\r k_1}E_{k_1}}{e^{\r(k+1)}}+\dfrac{M}{e^{\r (k+1)}}  \sum\limits_{j=k_1}^k \left( \dfrac{\dot{\b}_{j+1}}{\b_{j+1}}e^{\r (j+2)}-\dfrac{\dot{\b}_j}{\b_j}e^{\r (j+1)}\right)\\
&=& \dfrac{e^{\r k_1}E_{k_1}}{e^{\r(k+1)}}+\dfrac{M}{e^{\r (k+1)}} \left( \dfrac{\dot{\b}_{k+1}}{\b_{k+1}}e^{\r (k+2)}-\dfrac{\dot{\b}_{k_1}}{\b_{k_1}}e^{\r (k_1 +1)}\right)\\
& \leq&\dfrac{e^{\r k_1}E_{k_1}}{e^{\r(k+1)}} + Me^{\r }\dfrac{\dot{\b}_{k+1}}{\b_{k+1}}.
\end{array}
\end{equation*}
Consequently, for $k$ large enough
\begin{equation}\label{estim E_k}
E_{k}=  \mathcal{O}\left( \dfrac{\dot{\b}_{k}}{\b_{k}} + e^{-\r k}  \right).
\end{equation} 

\textbf{i)} Return  to \eqref{con- val disr} and using boundedness of the sequence $ \left( \frac{\dot{\b}_k}{\b_k} \right) _k$
we have $(E_k)$ is bounded, and thus
\begin{equation*}
f(x_k)-\underset{\mathcal{H}}{\min} \,f=\mathcal{O}\left( \dfrac{1}{\b_k} \right).
\end{equation*}
Combining \eqref{conv-tra discr} and \eqref{11d} with \eqref{estim E_k} we get  for $k$ large enough
\begin{equation}\label{32d}
\norm{\dot x_k }^2   = \mathcal{O} \left(    \dfrac{\dot{\b}_k}{\b_k}+e^{-\r k} \right)  \text{ and }\norm{ x_k -  y_{k} }^2   = \mathcal{O} \left(    \dfrac{\dot{\b}_k}{\b_k}+e^{-\r k} \right) .
\end{equation}

\textbf{ii)}  We have 
$$
\lim_{k\rightarrow +\infty}\dfrac{\dot{\b}_k}{\b_k}= \lim_{k\rightarrow +\infty}\dfrac{{\b}_{k+1}}{\b_k}-1=\ell-1=0,
$$
then $x_k -  y_{k} $ strongly converges to the origine of $\mathcal H$. Return to \eqref{x^*}, we conclude that $x_k$ strongly converges to $x^*$.\\

\textbf{iii)} According to \eqref{disc syst }, let $z_k\in \partial f(x_{k+1})$ satisfying $x_{k+1}+\b_kz_k-dx_{k}=0$. Then
\begin{eqnarray*}
d(0,\partial f(x_{k+1}))^2 &\leq& \|z_k\|^2 = \norm{\frac{d}{\b_k}\dot x_k + \frac{1-d}{\b_k}x_{k+1}}^2\\
	&\leq& \frac{2d^2}{\b_k^2}\norm{\dot x_k}^2 + \frac{2(1-d)^2}{\b_k^2}\norm{ x_{k+1}}^2.
\end{eqnarray*}
Hence, \eqref{32d} and boundedness of $(x_k)$ lead to 
$$
d(0,\partial f(x_{k+1}))= \mathcal{O} \left(    \dfrac{1}{\b_k}\right), \text{  for $k$ large enough}.
$$
If in addition $f$ is differentiable,   we obtain the large convergence rate \eqref{eq:grad35}.
\if{

to Lipschitz continuity of $\n f $, we rely on \eqref{nn} and \eqref{x^*} to obtain
\begin{eqnarray*}
\|\n f (x_k)\|^2 &\leq & 2 \|\n f (x_k) - \n f (y_{k})\|^2 + 2\| \n f (y_{k})\| ^2\\
	&\leq & 2L^2\|x_k - y_{k}\|^2 + \frac{2c^2}{\b^2_k}\|y_{k}\|^2\\
		&\leq & 2L^2\|x_k - y_{k}\|^2 + \frac{2c^2}{\b^2_k}\|x^*\|^2.
\end{eqnarray*}
Then using \eqref{32d}, we conclude
\begin{equation*}
\norm{ \nabla f(x_k)} ^2 =   \mathcal{O}\left( \dfrac{\dot{\b}_k}{\b_k}+ \dfrac{1}{\b ^2 _k} +e^{-\r k}\right),\,\, \textit{as} \; k \rightarrow +\infty. 
\end{equation*}
}\fi
\end{proof}

\section{Application to special cases}
In this section we review two special cases for the sequence $(\b_k )$.
\subsection{\underline {Case $\b_k = k^m \ln ^q (k) $}:}
 To  investigate the fulfillment of the condition  \ref{H02}, let us first note that for $a>0$ and for $k$ large enough we have 
$$
\;\left(1+\frac{a}{k}\right)^m =1+\frac{am}{k}+o\left( \frac{1}{k} \right)\;\; \text{ and } \;\;\ln \left(1+\frac{a}{k}\right) =\frac{a}{k}+o\left( \frac{1}{k} \right). 
$$
Then
\begin{eqnarray*}
\left(1+\frac{a}{k}\right)^m \left(1+\frac{\ln (1+\frac{a}{k})}{\ln (k)}\right) ^q &=&\left(1+\frac{am}{k}+o\left( \frac{1}{k} \right)\right) \left(1+\frac{a}{k\ln (k)}+o\left( \frac{1}{k\ln (k)} \right)\right) ^q\\\\
&=& \left(1+\frac{am}{k}+o\left( \frac{1}{k} \right)\right) \left(1+\frac{aq}{k\ln (k)}+o\left( \frac{1}{k\ln (k)} \right)\right);
\end{eqnarray*}
which gives 
\begin{equation}\label{Tay}
\left(1+\frac{a}{k}\right)^m \left(1+\frac{\ln (1+\frac{a}{k})}{\ln (k)}\right) ^q= 1+\frac{am}{k}+\frac{aq}{k\ln (k)}+o\left( \frac{1}{k\ln (k)} \right).
\end{equation}
\begin{itemize}
\item[$\bullet$]  For \ref{H02}(i), we have $\dot{\b}_ k\neq 0\textit{ for } k>1 $, which  is trivial.\\
\item[$\bullet$]  For \ref{H02}(ii), we come back to \eqref{Tay} to deduce that when $k$ is large enough
\begin{eqnarray*}
 \dfrac{\b_{k+1}}{\b_k}
&=& \dfrac{(k+1)^m \ln ^q  (k+1)}{k^m \ln ^q  (k)}\\
&=&\left( 1+\dfrac{1}{k} \right)^m \left(1+\dfrac{\ln (1+\frac{1}{k})}{\ln (k)} \right)^q\\
&=&1+\frac{m}{k}+\frac{q}{k\ln (k)}+o\left( \frac{1}{k\ln (k)} \right).
\end{eqnarray*}
Thus 
$\lim_{k\rightarrow +\infty}\dfrac{\b_{k+1}}{\b_k}=1>0$.
\\
We also  remark that
\begin{equation*}
\begin{array}{lll}
\dfrac{\dot{\b }_{k+1}}{\dot{\b}_k} &=&\dfrac{(k+2)^m\ln^q (k+2)-(k+1)^m\ln^q (k+1)}{(k+1)^m\ln^q (k+1)-k^m\ln^q(k)}\\\\
&=& \dfrac{\left(1+\frac{2}{k}\right)^m \left(1+\frac{\ln (1+\frac{2}{k})}{\ln (k)}\right) ^q-\left(1+\frac{1}{k}\right)^m\left(1+\frac{\ln (1+\frac{1}{k})}{\ln (k)}\right) ^q}{\left(1+\frac{1}{k}\right)^m\left(1+\frac{\ln (1+\frac{1}{k})}{\ln (k)}\right) ^q-1}.
\end{array}
\end{equation*}  
Then, using \eqref{Tay}, we get for $k$ large enough
\begin{equation*}
\begin{array}{lll}
\dfrac{\dot{\b }_{k+1}}{\dot{\b}_k} &=& \dfrac{\!\left(\!1+\frac{2m}{k}+\frac{2q}{k\ln (k)}+o\left(\! \frac{1}{k\ln (k)} \!\right) \right)-\left(\! 1+\frac{m}{k}+\frac{q}{k\ln (k)}+o\left( \frac{1}{k\ln (k)} \!\right)\!\right) }{\left(\!1+\frac{m}{k}+\frac{q}{k\ln (k)}+o\!\left(\! \frac{1}{k\ln (k)} \!\right)\!\right) -1}\\
&=& \dfrac{m\ln (k) +q +o(1) }{m\ln (k) +q +o(1)}.
\end{array}
\end{equation*} 
Thus, for $(m,q)\in (\mathbb{R}^+)^2 \setminus \{(0,0)\}$, we deduce $\underset{k \to +\infty}{\lim}\dfrac{\dot{\b }_{k+1}}{\dot{\b}_k}=1. $
\end{itemize}
Consequently all conditions of \ref{H02} are satisfied. 
So, based on Theorem \ref{Th 2 disc}, the following result can be proved  easily.
\begin{proposition}\label{corollary 1 ln}
Let $f$  and $(x_k)$ be as in Theorem \ref{Th 2 disc}, and $\b_k = k^m \ln ^q  (k) $ where $(m,q)\in (\mathbb{R}^+)^2 \setminus \{(0,0)\}$. Then $(x_k)$  converges strongly to $x^*$  the element of minimum norm of $\text{argmin}_{\mathcal{H}}f$ and
 \begin{eqnarray*}
&& f(x_k)-\underset{\mathcal{H}}{\min}f = \mathcal{O} \left( \frac{1}{k^m \ln^q(k)} \right);\;\;\;\;
\norm{ \dot x_k}^2   =\left \{
\begin{array}{lcl}
 \mathcal{O} \left( \frac{1}{k} \right) \;\;&\textit{ if }\; r\neq 0&  \\
 \mathcal{O} \left( \frac{1}{k\ln (k)} \right)\;\; &\textit{ if }\; r= 0;& \\
\end{array}
\right. \\ 
&&d(0,\partial f(x_{k}))=\mathcal{O} \left( \frac{1}{k^m \ln^q(k)} \right).
 \end{eqnarray*}
If moreover  $f$ is differentiable, then\; 
\begin{eqnarray*}
&&\; \norm{ \nabla f(x_k)}  =   
 \mathcal{O} \left( \dfrac{1}{k^{m}\ln ^{q}(k)} \right).
 \end{eqnarray*}
\end{proposition}
\subsection{\underline {Case $ \beta_k = k^me^{\gamma k^r}$}:\;}
  Let us  now  treat   the case $ \beta_k = k^me^{\gamma k^r}$ with $r\in ]0,1], m\geq 0$ and $\gamma >0$.  To  fulfill   the condition  \ref{H02}, we first remark that for $0<r<1$ and $a>0$ we have  for $k$ large enough
\begin{equation*}\label{Tay 2}
(k+a)^r -k^r = k^r \left[ \left( 1+\dfrac{a}{k} \right)^r -1  \right] =ark^{r-1}+o(k^{r-1}).
\end{equation*} 
Hence,   for $k$ large enough
\begin{equation}\label{Tay 3}
e^{\g((k+a)^r-k^r)}=e^{a\g r k^{r-1}+o(k^{r-1})}=1+a\g rk^{r-1}+o(k^{r-1}).
\end{equation}
\begin{itemize}
\item[$\bullet\,$] Likewise \ref{H02}(i), $\dot{\b}_ k\neq 0\textit{ for } k>1 $ is trivial.
\item[$\bullet\,$] For \ref{H02}(ii), we distinguish two cases:

$\star$ If $0<r<1$  then using \eqref{Tay 3} with $a=1$, we get for $k$ large enough 
 \begin{equation*}
 \dfrac{\b_{k+1}}{\b_k}= \left(1+\frac1{k}\right)^me^{\gamma \left((k+1)^r-k^r\right)}= 1+\g rk^{r-1}+o(k^{r-1}),
\end{equation*}  
\begin{equation*}
\begin{array}{lll}
\dfrac{\dot{\b }_{k+1}}{\dot{\b}_k}&=& \dfrac{(k+2)^me^{\g(k+2)^r}-(k+1)^me^{\g(k+1)^r}}{(k+1)^me^{\g(k+1)^r}-k^me^{\g k^r}}\\\\
&=&\dfrac{1+\frac{m}{\g rk^r}+o\left(\frac{m}{k^{-r}}\right)}{1+\frac{m}{\g rk^r}+o\left(\frac{m}{k^{-r}}\right)}.
\end{array}
\end{equation*}
Thus 
$$ 
\underset{k \to +\infty}{\lim}\dfrac{{\b }_{k+1}}{{\b}_k}=\underset{k \to +\infty}{\lim}\dfrac{\dot{\b }_{k+1}}{\dot{\b}_k}=1>0. 
$$

$\star$  If $r=1$, we have for $k$ large enough the termes  $ \dfrac{\b_{k+1}}{\b_k}$  and $\dfrac{\dot{\b }_{k+1}}{\dot{\b}_k}$ are equivalent to $e^{\g}\left( 1+\frac{m}{k}+o\left(\frac{1}{k}\right)\right)$, so 
$$
\underset{k \to +\infty}{\lim}\dfrac{{\b }_{k+1}}{{\b}_k} =\underset{k \to +\infty}{\lim}\dfrac{\dot{\b }_{k+1}}{\dot{\b}_k}=e^{\g}>0. 
$$
\end{itemize} 

Applying again  Theorem \ref{Th 2 disc}, we get the following Corollary.
\begin{proposition}\label{corollary 2 disc}
Let $f$  and $(x_k)$ be as in Theorem \ref{Th 2 disc}, where $\b_k =   k^me^{\gamma k^r}.$ 
Then

$\star$  if $r=1$ and $m=0$, we have  $(x_k)$ is bounded and   for $k$ large enough
\begin{eqnarray*}
&&f(x_k)-\underset{\mathcal{H}}{\min}f = \mathcal{O} \left( e^{-\gamma k} \right);
\end{eqnarray*}

$\star$ if either $0<r<1$ or $m>0$,  $(x_k)$  converges strongly to $x^*$  the element of minimum norm of $\text{argmin}_{\mathcal{H}}f$, and for $k$ large enough, we have   
\begin{eqnarray*}
&& f(x_k)-\underset{\mathcal{H}}{\min}f = \mathcal{O} \left( k^{-m}e^{-\gamma k^r} \right), \;
\norm{\dot x_k}^2   = \mathcal{O} \left( \dfrac{1}{k^{1-r}} \right)\;\\
&&d(0,\partial f(x_{k}))=  \mathcal{O} \left( k^{-m}e^{-\gamma k^r} \right).
 \end{eqnarray*}
If moreover  $f$ is differentiable, then\; 
\begin{eqnarray*}
\text{and }\; \norm{\n f(x_{k})}  =    \mathcal{O} \left( k^{-m}e^{-\gamma k^r} \right).
\end{eqnarray*}
\end{proposition}

\section{Numerical example}
In this section, we present a numerical example in the framework of nondifferentiable minimization problem to illustrate the performance of our iterative method \eqref{proximal algorithm}. So, we consider the proper lower semicontinuous and convex function 
$$
f(x,y)= |x| + \iota_{[-a,a]}(x) + \frac12(y-v_0)^2
$$
where $a>0, v_0\in \mathbb R$ and $\iota_{[-a,a]}$ is the indicator function having the value $0$ on $[-a,a]$ and $+\infty$ elsewhere.   The well-known Fermat's and subdifferential sum rules for convex functions ensure that
$
(\bar x,\bar y) \in\text{argmin}_{\mathbb R^2}f ,
$
iff 
$$
(0,0)\in \partial f(\bar x,\bar y) = \partial\left( |\cdot| + \delta_{[-a,a]}\right)(\bar x)\times \partial\left( \frac12(\cdot-v_0)^2\right)(\bar y)  \iff (\bar x,\bar y)=(0,v_0).
$$
Thus $\text{argmin}_{\mathbb R^2}f = \{0,v_0)\}$.
Using the rules for calculating proximal operators, see \cite[Chap 6]{beck}, we have  for each $\lambda >0, (x,y)\in \mathbb R^2$,
\begin{eqnarray*}
\text{prox}_{\lambda f}(x,y) &=& \left( \text{prox}_{\lambda \left( |\cdot| + \delta_{[-a,a]}\right)}(x)\, , \, \text{prox}_{\lambda \left( \frac12(\cdot-v_0)^2\right)}(y)\right)\\
	&=&  \left(\min[\max(|x|-\l,0),a]\text{sign}(x)\; , \; -\frac1{\l + 1}(y-\l v_0)\right)
\end{eqnarray*}
Below, we explain our algorithm \eqref{proximal algorithm} and the one proposed by L\'aszl\'o in the recent paper \cite{Las23}:
\begin{eqnarray}
\label{eq:BCR}&&\left\{
\begin{array}{rll}
(x_0,y_0)&\in& \mathbb R^2,\\
(x_{k+1},y_{k+1})&=&\text{prox}_{\b_k f}(d x_k,dy_k)
\end{array}\right.
\\
\label{eq:L}&&\left\{
\begin{array}{rll}
(x_0,y_0)&,& (x_1,y_1)\in\mathbb R^2,\\
(u_k,v_k)&=& (x_k,y_k)+(1-\frac{\a}{k^q})(x_k-x_{k-1},y_k-y_{k-1}),\\
(x_{k+1},y_{k+1})&=&\text{prox}_{\l_k f}\left((u_k,v_k)-\frac{c}{k^p}(x_k,y_k)\right)
\end{array}\right.
\end{eqnarray}
We note that \eqref{eq:L} is an implicit discretization of the differentiable dynamical system studied in \cite{Las23A}, that is $\ddot x(t)+ \frac{\alpha}{t^q} \dot x(t)+\nabla f(x(t))+ \frac{c}{t^p} x(t)=0$ where $\alpha,c, q,p>0$. 
In \cite[Theorem 1.1]{Las23}, for $\l_k=\l k^\delta$, the conditions on the parameters $p,q,\alpha, \delta, \l$ and $c$ impose that $0< q<1, p\leq 2, \l>0, \d\in\mathbb R, c>0$ where the choice of $\d,\l$ and $c$ depend on the positioning of $p$ with respect to $q +1$.
\begin{figure} 
 \includegraphics[scale=0.35]{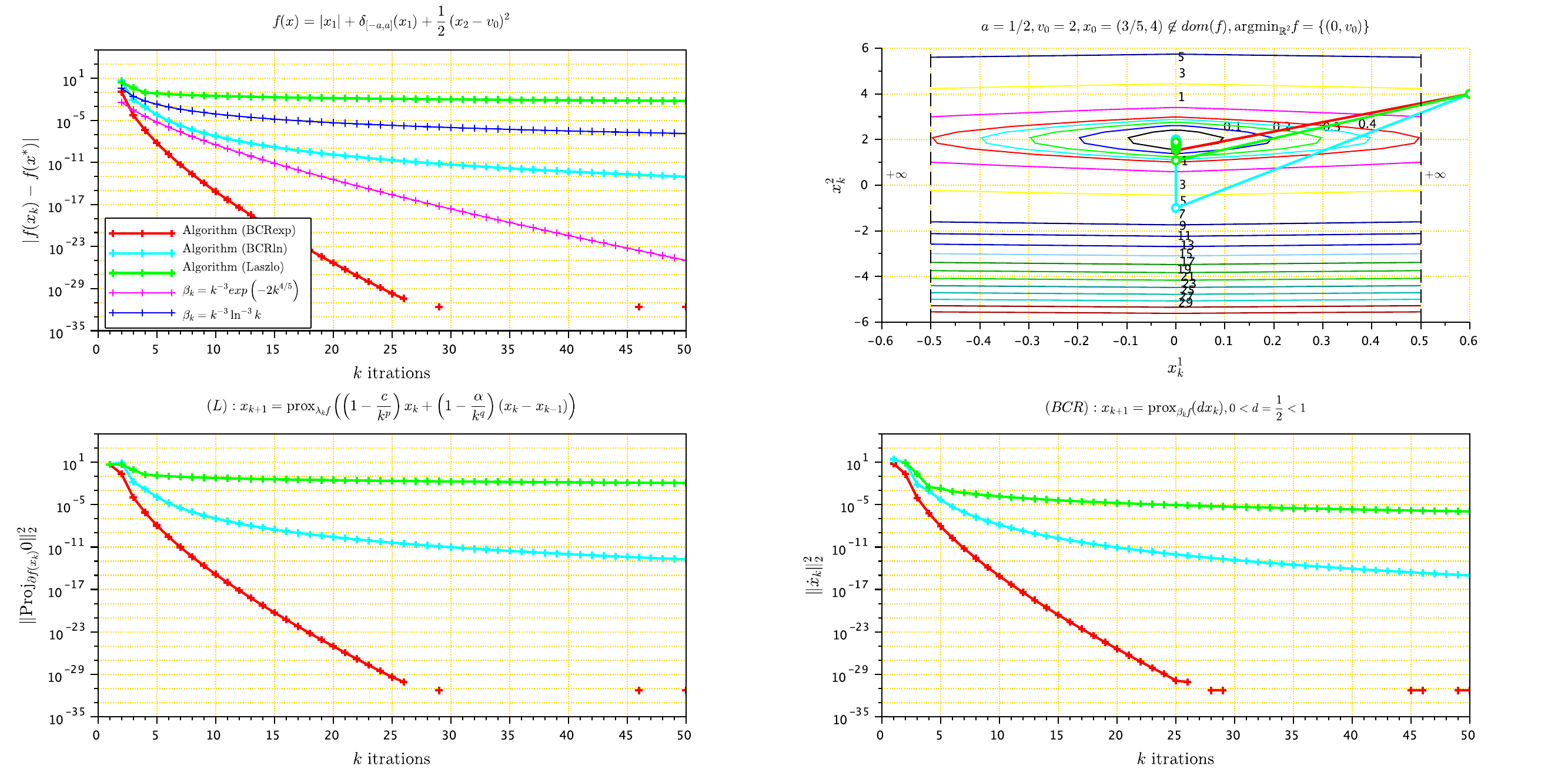}
  \caption{Comparison of convergence rates between the \eqref{eq:BCR} algorithm and that of L\'aszl\'o.}
 \label{fig:trigs-d} 
\end{figure}
The squared distance of $(x_k,y_k)$, generated by the above two iterations, to the  solution $(x^*,y^*)$ and the decay of the objective function $f(x_k,y_k)-\min_{\mathbb R^2}f(x,y)$ along the iterations are shown in Figure \ref{fig:trigs-d}  for the selected values $\b_k$ equal respectively to $k^3\ln^3k$ and $k^3e^{2 k^{4/5}}$ for the algorithm \eqref{eq:BCR}, and those equal to $p=2, q=\frac45, \alpha=2, \delta=2, \l=5$ for the algorithm \eqref{eq:L}. 
We note that the choice of $\beta_k=k^3e^{2 k^{4/5}}$ justifies the originality of  Theorem \ref{Th 2 disc}, since  we end up with an exponential convergence rate for the values and the gradient. Also,  the benefit of the inverse-relaxation $\beta_k=k^3e^{2 k^{4/5}}$ (red curve  in Figure \ref{fig:trigs-d}), as allowed in Proposition \ref{corollary 2 disc}, is clearly visible.

\begin{figure} 
 \includegraphics[scale=0.35]{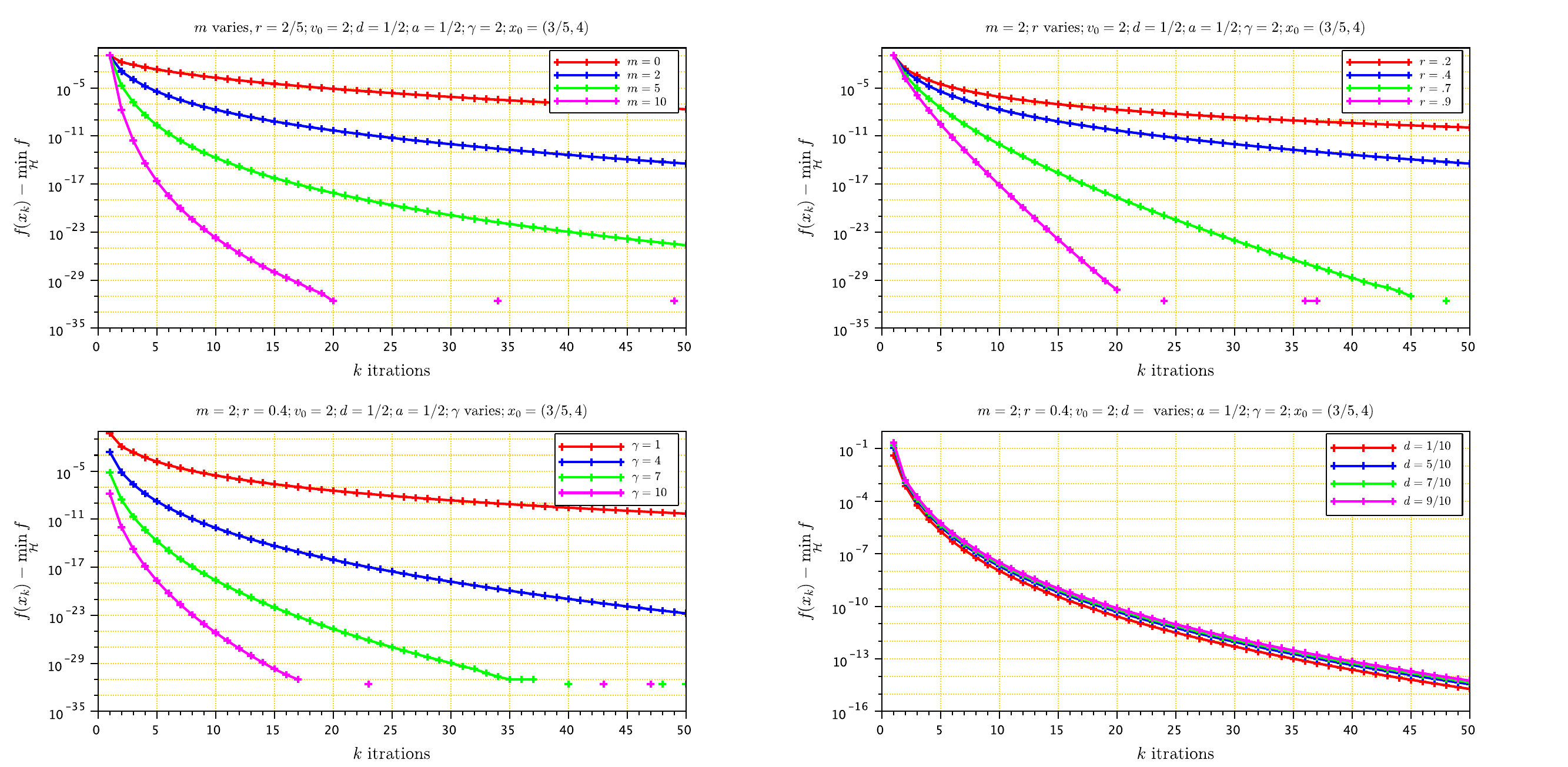}
  \caption{Convergence rates of values, trajectories  and gradients.}
 \label{fig:trigs-param} 
\end{figure}
Figure \ref{fig:trigs-param} explains the interest in the growth of the three parameters $m,r,\gamma$ in $\beta_k=k^me^{\gamma k^r}$, while the parameter $d$ in Algorithm \eqref{proximal algorithm} does not act on the convergence rates.
\if{
\newpage
\section{Implicit discretization for nonsmooth case}

\subsection{From prox algorithm to that of Forward-Backward}\label{sec:numerics}
Here, we illustrate our results on the composite problem on $\cH=\R^n$,
\begin{equation}\label{eq:minP}\tag{RLS}
\min_{x \in \R^n} \left\{ f(x) = \frac{1}{2}\norm{y-Ax}^2 + g(x) \right\} ,
\end{equation}
where $A$ is a linear operator from $\R^n$ to $\R^m$, $m \leq n$, $g: \R^n \to \Rb$ is a proper lower semicontinuous convex function which acts as a regularizer. The results in this section should usually be understood in finite dimensions, but most of them are dimension-independent and often hold in a Hilbert space.

Problem~\eqref{eq:minP} is extremely popular in a variety of fields ranging from inverse problems in signal/image processing, to machine learning and statistics. Typical examples of $g$ include the $\ell_1$ norm (Lasso), the $\ell_1-\ell_2$ norm (group Lasso), the total variation, or the nuclear norm (the $\ell_1$ norm of the singular values of $x \in \R^{N \times N}$ identified with a vector in $\R^n$ with $n=N^2$). To avoid trivialities, we assume that the set of minimizers of~\eqref{eq:minP} is non-empty.

Let $M=s^{-1}I-A^*A$, then proviso that $0 < s\norm{A}^2 < 1$, $M$ is a symmetric positive definite matrix. We define $\norm{x}_M =\langle Mx,x\rangle^{1/2}$, which in finite-dimensional spaces is a norm equivalent to $\norm{x}$. Then the associate Moreau envelope and proximal mapping are defined by
\[
f_{\lambda M}(x) = \min_{z \in \R^n} \left(\frac{1}{2\lambda }\norm{z-x}_M^2 + f(z)\right) = \min_{z \in \R^n} \left(\frac{1}{2}\norm{z-x}_{\lambda M}^2 + f(z)\right), 
\] 
\[
 {\prox}^{\lambda M}_{f}(x) =  {\prox}^{M}_{\lambda f}(x) =
\underset{z \in \R^n}{\argmin }
 \left(\frac{1}{2\lambda }\norm{z-x}_M^2 + f(z) \right).
\] 
Thus, our proposed algorithm is
\begin{equation}\label{eq:proxM}
x_{k+1}=\text{prox}_{\b_k f}^M(d x_k).
\end{equation}
We therefore apply Proposition \ref{corollary 1 ln} to the associate algorithm \eqref{eq:proxM} to obtain
\begin{proposition}\label{thm:M1}
Suppose $f$ proper convex lower semicontinous and $(x_k)$ generated by  \eqref{eq:proxM} such that $\b_k = k^r \ln ^q  (k) $ where $r,q\in (\mathbb{R}^+)^2 \setminus \{(0,0)\}$. Then $(x_k)$  converges strongly to $x^*$  the element of minimum norm of $\text{argmin}_{\mathcal{H}}f$ with respect to $\norm{\cdot}_M$, and
 \begin{eqnarray*}
&& f(x_k)-\underset{\mathcal{H}}{\min}f = \mathcal{O} \left( \frac{1}{k^r \ln^q(k)} \right);\;\;\;\;
\norm{ \dot x_k}^2   =\left \{
\begin{array}{lcl}
 \mathcal{O} \left( \frac{1}{k} \right) \;\;&\textit{ if }\; r\neq 0&  \\
 \mathcal{O} \left( \frac{1}{k\ln (k)} \right)\;\; &\textit{ if }\; r= 0.& \\
\end{array}
\right. 
 \end{eqnarray*}
\end{proposition}
\begin{proof}
Since  in finite-dimensional spaces all norms are equivalent to $\norm{x}$, then  in theory the convergence rate for $\dot x(t)$ is unchanged with respect to any norm in $\mathbb R^n$, while the projection of the origin onto  $\text{argmin}_{\mathcal{H}}f$  depends on the chosen norm.
\end{proof}
Likewise, applying Proposition \ref{corollary 2 disc}, we have
\begin{proposition}\label{thm:M2}
Let $f$  and $(x_k)$ be as in Theorem \ref{thm:M1}, where $\b_k =  e^{\gamma k^r}.$ 
Then

$\star$  for $r=1$, $(x_k)$ is bounded and  $f(x_k)-\underset{\mathcal{H}}{\min}f = \mathcal{O} \left( e^{-\gamma k} \right)$;

$\star$ for $0<r<1$,  $(x_k)$  converges strongly to the element of minimum norm of $\text{argmin}_{\mathcal{H}}f$ with respect to $\norm{\cdot}_M$, and 
\begin{eqnarray*}
&& f(x_k)-\underset{\mathcal{H}}{\min}f = \mathcal{O} \left( e^{-\gamma k^r} \right), \;
\norm{\dot x_k}^2   = \mathcal{O} \left( \dfrac{1}{k^{1-r}} \right).
\end{eqnarray*}
\end{proposition}
We avoid here giving the estimate of the norm of the gradient $\nabla f$, because the function $ f(x) = \frac{1}{2}\norm{y-Ax}^2 + g(x) $ cannot be differentiable if $g$ is not so.

It can be easily shown (see also the discussion in~\cite[Section~4.6]{chambollereview}), that the proximal mapping of $f$ as defined in \eqref{eq:minP} with the norm $\norm{\cdot}_M$ is
\begin{equation*}\label{eq:proxFBM}
{{\prox}}^M_{f}(x) = {\prox}_{s g}(x - s A^*(Ax-y)) .
\end{equation*}
Thus, our algorithm $x_{k+1}=\text{prox}_{\b_k f}^M(d x_k)$ becomes 
\begin{equation}\label{eq:ForBack}
x_{k+1}= {\prox}_{s g}(d x_k - s A^*(dA x_k-y)),
\end{equation}
which is a kind of forward-backward algorithms for the objective function in \eqref{eq:minP}.

Moreover, $f_M$ is a continuously differentiable convex function whose gradient (again in the metric $M$) is given by the standard identity
\[
\nabla f_M(x) = x - {\prox}^M_{f}(x) ,
\]
and   $\nabla f_M$ is 1-Lipschitz continuous in the metric $M$. In addition, a standard argument shows that 
\[
\argmin_{\cH} f = \mathrm{Fix}({\prox}^M_{f}) = \argmin_{\cH} f_M .
\]
We are then in position to solve \eqref{eq:minP} by simply applying IGAHD (see Section~\ref{sec:igahd}) to $f_M$. We infer from Theorem~\ref{pr.decay_E_k} and properties of $f_M$ that
\[
f({\prox}^M_f(x_k))-\min_{\R^n} f = \mathcal{O}(k^{-2}) .
\]
IGAHD and FISTA (\ie IGAHD with $\beta=0$) were applied to $f_M$ with four instances of $g$: $\ell_1$ norm, $\ell_1-\ell_2$ norm, the total variation, and the nuclear norm. The results are depicted in Figure~\ref{fig:rls}. One can clearly see that the convergence profiles observed for both algorithms agree with the predicted rate. Moreover, IGAHD exhibits, as expected, less oscillations than FISTA, and eventually converges faster.

\subsection{}
Let us extend the results of Theorem \ref{Th 2 disc} to the case of a proper lower semicontinuous
and convex function $f: \mathcal{H} \longrightarrow \mathbb{R}\cup \{+\infty\}$. 
\if{Referring to   \cite{Brezis , BaCo}, we rely on the basic properties of the Moreau envelope $f_{\g}: \mathcal{H} \longrightarrow \mathbb{R}  \;\; (\g >0)$, which is defined by
\begin{equation}\label{1s}
f_\g (x) = {\min}_{y \in \mathcal{H}} \left\lbrace f(y)+\frac{1}{2\g }\norm{x-y}^2 \right\rbrace \;\; \text{for any } x\in \mathcal{H}.
\end{equation}
Recall that,  the functions $f$ and $f_\l$ share the same optimal objective value and the same
set of minimizers
$$ {\min}_{y \in \mathcal{H}}\, f = {\min}_{y \in \mathcal{H}}\, f_\g \;\;\;\;\;\; \text{and}\;\;\;\;\;\; {\argmin}_{\mathcal{H}}\, f ={\argmin}_{\mathcal{H}}\, f_\g . $$
In addition, $f_\l$ is convex and continuously differentiable whose gradient is $\frac{1}{\g}$-Lipschitz continuous.

The unique point where the minimum value is achieved in \eqref{1s} is denoted by ${\prox} _{\g f} (x) $, and satisfies the following classical formulas: For each $\g >0, x\in \mathcal{H}$, 
\begin{itemize}
\item $ f_\g (x)=f({\prox} _{\g f} (x)) +\frac{1}{2\g}\norm{x-{\prox} _{\g f} (x)}^2  $;
\item $\nabla f_\g (x) = \frac{1}{\g} (x-{\prox} _{\g f} (x)) $;
\item ${\prox} _{\theta f_\g}(x)=\frac{\g}{ \g +\theta} x + \frac{\theta}{ \g +\theta} {\prox} _{(\g + \theta)f}(x), \;\text{ for all } \theta > 0.$
\end{itemize}
}\fi
Since the set of minimizers and the infimal value are preserved by taking the Moreau envelope, see Appendix \ref{sec-appendix}, the idea is to replace $f$ by $f_\g$ in the algorithm \eqref{proximal algorithm}. Then, \eqref{proximal algorithm} applied to $f_\g$ now reads 
$$
 x_{k+1}=\text{prox}_{\b_k f_\g} (d x_k).
$$
By using Lemma \ref{lem-basic-c}(iii), we get for $\g>0$ the following relaxed inertial proximal algorithm\\
\begin{equation}\label{prox_NS}
x_{k+1}= \frac{\g d}{\g +\b_k}x_k +\frac{\b_k}{\g +\b_k} \text{prox}_{(\g+\b_k) f}(d x_k).
\end{equation}
Let us compute the viscosity curve associated with $f_\g$, denoted by $x_{\g , \b}$.

Recall that
$
x_{\beta}  =  {\argmin}_{ \mathcal{H}} \left( f+\frac{1-d}{2\beta}\norm{.}^2 \right)  = {{\prox}} _{\frac{\beta}{1-d}f} (0) $.
Then
\begin{eqnarray}
x_{\g , \beta} &=& {{\prox}}_{\frac{\beta}{1-d}f_\g} (0) \nonumber \\
&=& \frac{\beta}{\beta+\g (1-d) }  \, {{\prox}}_{\frac{\beta+\g (1-d)}{1-d}  f} (0) \nonumber\\
&=& \frac{\beta}{\beta+\g (1-d) }  \, x_{\beta +\g (1-d)}.	\label{3s}
\end{eqnarray}
We now have all the ingredients to apply Theorem \ref{Th 2 disc} to \eqref{prox_NS}, and so we get the following result.
\begin{theorem}\label{Th1 ns}
Let $f: \mathcal{H} \longrightarrow \mathbb{R}\cup \{+\infty\}$ be a proper, lower semicontinuous and convex function.  Let $(x_k)$ be a sequence generated by the algorithm \eqref{prox_NS}.\\
Suppose $\b_k$ satisfies all conditions of Theorem \ref{Th 2 disc}. Then, for $k$ large enough, we have
\begin{eqnarray}
&&f ({{\prox}}_{\g f}(x_k))-\underset{\mathcal{H}}{\min}f = \mathcal{O} \left(   \frac{1}{\b_k} \right),	\label{4s}\\
&&\norm{ x_k-  {{\prox}}_{\g f}(x_k) } ^2  = \mathcal{O} \left(    \dfrac{1}{\b_k}\right), 	\label{6s}\\
&&\norm{ x_k- \frac{\beta _k}{\beta _k+\g (1-d)} \,x_{\beta _k +\g (1-d)} }  ^2 =  \mathcal{O} \left(     \dfrac{\dot{\beta}_{k}}{\beta _{k}}+e^{-\r k}\right) .	\label{7s}
\end{eqnarray}
We also conclude that  $(x_k)$ strongly converges to the minimum norm element of $\text{argmin}_{\mathcal{H}}f$.
\end{theorem}
\begin{proof}
By applying Theorem \ref{Th 2 disc} (i) and (ii) to the function $f_\g$, we get for $k$ large enough
\begin{eqnarray}
&&f_\g(x_k)-\underset{\mathcal{H}}{\min}f_\g = \mathcal{O} \left( \frac{1}{\b_k} \right), \label{8s}\\
&&\norm{ x_k-  x_{\g ,\b_k} }  ^2 = \mathcal{O} \left(     \dfrac{\dot{\beta}_{k}}{\beta _{k}}+e^{-\r k} \right). \label{9s}
\end{eqnarray}
Since ${\min}_{\mathcal{H}}\,f_\g = {\min}_{\mathcal{H}}\,f $, from \eqref{8s} we deduce the following estimate
\begin{equation}\label{10s}
f_\g(x_k)-\underset{\mathcal{H}}{\min}f = \mathcal{O} \left( \frac{1}{\b_k} \right).
\end{equation}
Using
\begin{equation}\label{60}
 f_\g (x_k)- \underset{\mathcal{H}}{\min}f = \left( f\left({{\prox}} _{\g f} (x_k)\right)- \underset{\mathcal{H}}{\min}f  \right) +\frac{1}{2\g}\norm{x_k-{{\prox}} _{\g f} (x_k)}^2 ,
\end{equation}
we conclude \eqref{4s}, since
\begin{equation*}
0\leq f( {{\prox}} _{\g f} (x_k))- \underset{\mathcal{H}}{\min}f \leq  f_\g (x_k)- \underset{\mathcal{H}}{\min}f.
\end{equation*}
Again combining \eqref{8s} and  \eqref{60}, we obtain  \eqref{6s}.
Also, according to \eqref{3s} and \eqref{9s}, we get  \eqref{7s}.
Now,
since
$$ 
\underset{k\longrightarrow +\infty}{\lim}\,\frac{\beta _k}{\beta _k+\g (1-d)}=1 \;\text{ and }\; \underset{k\longrightarrow +\infty}{\lim}\,\beta _k +\g (1-d)= +\infty  
$$
the estimation \eqref{7s} leads to strong convergence of  $\frac{\beta _k}{\beta _k+\g (1-d)} \, x_{\beta _k +\g (1-d)}$  to $\text{proj} _ {{\argmin}_{\mathcal{H}} f} (0)$, which in turn implies that $(x_k)$ strongly converges to the minimum norm solution.
\end{proof}
}\fi

\section{Appendix}\label{sec-appendix}
We rely on the basic properties of the Moreau envelope $f_{\g}: \mathcal{H} \longrightarrow \mathbb{R}  \;\; (\g >0)$, which is defined by
\begin{equation}\label{1s}
f_\g (x) = {\min}_{y \in \mathcal{H}} \left( f(y)+\frac{1}{2\g }\norm{x-y}^2 \right) \;\; \text{for any } x\in \mathcal{H}.
\end{equation}
Recall that,  the functions $f$ and $f_\l$ share the same optimal objective value and the same
set of minimizers
$$ {\min}_{y \in \mathcal{H}}\, f = {\min}_{y \in \mathcal{H}}\, f_\g \;\;\;\;\;\; \text{and}\;\;\;\;\;\; {\argmin}_{\mathcal{H}}\, f ={\argmin}_{\mathcal{H}}\, f_\g . $$
In addition, $f_\l$ is convex and continuously differentiable whose gradient is $\frac{1}{\g}$-Lipschitz continuous.
\begin{lemma}[{\rm\cite{Att2}, \cite[Section 12]{BaCo}}, \cite{Brezis}]\label{lem-basic-c}
The unique point where the minimum value is achieved in \eqref{1s} is denoted by ${\prox} _{\g f} (x) $, and satisfies the following classical formulas: For each $\g >0$ and $ x\in \mathcal{H}$, 
\begin{itemize}
\item[$(i)$] $ f_\g (x)=f({\prox} _{\g f} (x)) +\frac{1}{2\g}\norm{x-{\prox} _{\g f} (x)}^2  $;
\item[$(ii)$] $\nabla f_\g (x) = \frac{1}{\g} (x-{\prox} _{\g f} (x)) $;
\item[$(iii)$] ${\prox} _{\theta f_\g}(x)=\frac{\g}{ \g +\theta} x + \frac{\theta}{ \g +\theta} {\prox} _{(\g + \theta)f}(x), \;\text{ for all } \theta > 0$;
\item[$(iv)$] $  \|{\prox} _{\g f} (0)\|\leq \|x^{*}\|$, where  $x^{*}=\mbox{\rm proj}_{\argmin f} 0$; \label{2a}  \medskip
\item[$(v)$]  $\lim_{\g \rightarrow +\infty}\|{\prox} _{\g f} (0)-x^{*}\|=0 $. \label{2b}
\end{itemize}
 \end{lemma}
The following Lemma provides an extended  version of the classical gradient lemma which is valid for differentiable convex functions. The following version has been obtained in \cite[Lemma 1]{ACFR}, \cite{ACFR-Optimisation}.
We reproduce its proof for the convenience of the reader.

\begin{lemma}\label{ext_descent_lemma}
Let  $f: \cH \to \R$ be  a  convex function whose gradient is $L$-Lipschitz continuous. Let $s \in ]0,1/L]$. Then for all $(x,y) \in \cH^2$, we have
\begin{equation}\label{eq:extdesclem}
f(y - s \nabla f (y)) \leq f (x) + \left\langle  \nabla f (y), y-x \right\rangle -\frac{s}{2} \|  \nabla f (y) \|^2 -\frac{s}{2} \| \nabla f (x)- \nabla f (y) \|^2 .
\end{equation}
In particular, when $\argmin f \neq \emptyset$, we obtain that for any $x\in \cH$
\begin{equation}\label{eq:extdesclemb}
f(x) \geq \min_{\cH} f  +\frac{1}{2L} \| \nabla f (x)\|^2 .
\end{equation}
\end{lemma}

\begin{proof}
Denote $y^+=y - s \nabla f (y)$. By the standard descent lemma applied to $y^+$ and $y$, and since $sL \leq 1$ we have
\begin{equation}\label{eq:descfm2}
f(y^+) \leq f(y) - \frac{s}{2}\left(2-Ls\right) \| \nabla f (y) \|^2 \leq f(y) - \frac{s}{2} \|  \nabla f (y) \|^2.
\end{equation}
We now argue by duality between strong convexity and Lipschitz continuity of the gradient of a convex function. Indeed, using Fenchel identity, we have
\[
f(y) = \dotp{\nabla f(y)}{y} - f^*(\nabla f(y)) .
\]
$L$-Lipschitz continuity of the gradient of $f$ is equivalent to $1/L$-strong convexity of its conjugate $f^*$. This together with the fact that $(\nabla f)^{-1}=\partial f^*$ gives for all $(x,y) \in \cH^2$,
\[
f^*(\nabla f(y)) \geq  f^*(\nabla f(x)) + \dotp{x}{\nabla f(y)-\nabla f(x)} + \frac{1}{2L}\norm{\nabla f(x)-\nabla f(y)}^2 .
\]
Inserting this inequality into the Fenchel identity above yields
\begin{align*}
f(y) 
&\leq - f^*(\nabla f(x)) + \dotp{\nabla f(y)}{y} - \dotp{x}{\nabla f(y)-\nabla f(x)} - \frac{1}{2L}\norm{\nabla f(x)-\nabla f(y)}^2 \\
&= - f^*(\nabla f(x)) + \dotp{x}{\nabla f(x)} + \dotp{\nabla f(y)}{y-x} - \frac{1}{2L}\norm{\nabla f(x)-\nabla f(y)}^2 \\
&= f(x) + \dotp{\nabla f(y)}{y-x} - \frac{1}{2L}\norm{\nabla f(x)-\nabla f(y)}^2  \\
&\leq f(x) + \dotp{\nabla f(y)}{y-x} - \frac{s}{2}\norm{\nabla f(x)-\nabla f(y)}^2 .
\end{align*}
Inserting the last bound into \eqref{eq:descfm2} completes the proof.
\end{proof}

\if{Let us now recall the differentiability properties of the viscosity curve.
\begin{lemma}[{\rm\cite{Att2}, \cite{AttCom}}]\label{lem-basic-cc}
The function $\varepsilon \mapsto x_{\varepsilon}$ is 
 Lipschitz continuous on the compact intervals of $]0, +\infty[$, hence almost everywhere differentiable, and the following inequality holds:
 \begin{equation}\label{44}
\| \frac{d}{d\varepsilon}\left( x_{\varepsilon} \right)\| \leq \frac{\|x^{*}\|}{\varepsilon}.
\end{equation} 
 \end{lemma}
}\fi

\medskip


\begin{thebibliography}{10}

\bibitem{att83} {H. Attouch}, {Variational Convergence for Functions and Operators}, Appl. Math. Ser., Pitman Advanced Publishing Program, Boston, 1984.

\bibitem{Att2} {\sc H. Attouch},   {\em Viscosity solutions of minimization problems}, SIAM J. Optim. 6 (3) (1996), 769--806.


\bibitem{ABCR-JDE} 
\newblock H. Attouch, A. Balhag, Z. Chbani and H. Riahi,
\newblock {Damped inertial dynamics with vanishing Tikhonov regularization: Strong asymptotic convergence towards the minimum norm solution},
\newblock \emph{J. Differential Equations}, \textbf{311} (2022), 29-58.

\bibitem{ACRA} {\sc H. Attouch, A. Balhag, Z. Chbani, H. Riahi}, {\em Fast convex optimization via inertial dynamics combining viscous and Hessian-driven damping with time rescaling}, Evol.Equ. Control Theory 11 (2022), no. 2, 487-514.

\bibitem{ACFR} {\sc  H. Attouch, Z. Chbani, J. Fadili, H. Riahi}, First order optimization algorithms via inertial  systems with Hessian driven damping,  Math. Program.,  193 (2022),  113--155.


\bibitem{ACFR-Optimisation}{\sc H. Attouch,  Z. Chbani, J. Fadili, H. Riahi}, Convergence of iterates for first-order optimization algorithms with inertia and Hessian driven damping, Optimization, 72 (2023), 1199--1238.
 https://doi.org/10.1080/02331934.2021.2009828


 \bibitem{ACR1} {\sc H. Attouch, Z. Chbani, H. Riahi}, {\em Fast convex optimization via time scaling of damped inertial gradient dynamics}, SIAM J.Optim., 29 (3) (2019), 2227-2256.
 
  \bibitem{ACR2} {\sc H. Attouch, Z. Chbani, H. Riahi}, { Fast convex optimization via time scaling of damped inertial gradient dynamics},  11 (2022),   487--514.  Doi: 10.3934/eect.2021010 



\bibitem{BCR1} A. C. Bagy, Z. Chbani, H. Riahi, {\it The heavy ball method regularized by Tikhonov term. Simultaneous convergence of values and trajectories}, Evolution Equations and Control Theory,  \textbf{12} (2023), No. 2, 687--702.

\bibitem{AttCom} {\sc H. Attouch,  R. Cominetti},  {\em A dynamical approach to convex
minimization coupling approximation with the steepest descent method},   J.
Differential Equations, 128 (2) (1996), 519--540.

\bibitem{attcz17} {H. Attouch, M.-O. Czarnecki}, {\em Asymptotic behavior of gradient-like dynamical systems involving inertia and multiscale aspects},
Journal of Differential Equations 262 (3), 2745--2770

\bibitem{attcz10} {H. Attouch, M.-O. Czarnecki}, {\em Asymptotic behavior of coupled dynamical systems with multiscale aspects}, 
J. Differential Equations,  248 (2010) 1315--1344.	
https://doi.org/10.1016/j.jde.2009.06.014


 
\bibitem{BaCo}  H. Bauschke, P. L. Combettes, {\it Convex Analysis and Monotone Operator Theory in Hilbert Spaces}, CSM Books in Mathematics, Springer, 2011.

\bibitem{beck} A. Beck, First-Order Methods in Optimization, SIAM, Philadelphia : Mathematical Optimization Society, 2017.

\bibitem{BT} A. Beck, M. Teboulle,  {\it  A fast iterative shrinkage-thresholding algorithm for linear inverse problems},  SIAM J. Imaging Sci., \textbf{2 } (2009),  No. 1, 183--202.

 \bibitem{Brezis} {\sc H. Br\'ezis}, {\em Op\'erateurs Maximaux Monotones dans les Espaces de Hilbert et Équations D'évolution}, Lecture Notes, vol.5, North Holland, 1972.

\bibitem{Brezis2} {\sc H. Br\'ezis}, {\em  Analyse Fonctionnelle}, Masson, 1983.

\bibitem{cabo05}  {\sc Cabot A.}, {\em  Proximal point algorithm controlled by a slowly vanishing term: Applications to hierarchical minimization}, SIAM J. Optim. 15 (2005), 555--572.

\bibitem{ccr00} Chadli O, Chbani Z, Riahi H. Equilibrium problems with generalized monotone bifunctions and applications to variational inequalities. J. Optim. Theory Appl. 105 (2000), 299--323.

\bibitem{Guler}{\sc O. Güler}, {\em On the convergence of the proximal point algorithm for convex minimization}, SIAM J. Control Optim. 29 (1991), 403-419.


 \bibitem{Lahdili}{\sc N. Lehdili, A. Moudaf}, {\em Combining the proximal algorithm and Tikhonov regularization}, Optimization 37 (1996), 239-252.

\bibitem{Las23A}   S.C. L\' aszl\'o, 
{\it  On the strong convergence of the trajectories of a Tikhonov regularized second order dynamical system with asymptotically vanishing damping}, 
J. Differential Equations, \textbf{362} (2023), 355--381.

\bibitem{Las23}  S.C. L\' aszl\'o,  {\it  On the convergence of an inertial proximal algorithm with a Tikhonov regularization term}, 2023, https://doi.org/10.21203/rs.3.rs-2882874/v1

\bibitem{Martinet}{\sc B. Martinet}, {\em Régularisation d'inéquations variationnelles par approximations successives}, Rev. Franaise Informat. Recherche Opérationnelle 4 (1970), 154-158 (in
French).

\bibitem{martinet}  B. Martinet, {\it Br\`eve communication. R\'egularisation d'in\'equations variationnelles par approximations
successives}, 
 ESAIM: Mathematical Modelling and Numerical Analysis - Mod\'elisation Math\'ematique et Analyse Num\'erique,  \textbf{4} (1970),  No.3,  154--158.  

\bibitem{May}  R. May, {\it Asymptotic for a second-order evolution equation with convex potential and vanishing damping term},    Turkish Journal of Mathematics, \textbf{41} (2017), No. 3,  681--685.

\bibitem{Morozov}{\sc V. A. Morozov}, {\em Methods of solving incorrectly posed problems}, Springer Verlag, New York, 1984.


\bibitem{Nest1} Y. Nesterov, {\it   A method of solving a convex programming problem with convergence rate
O(1/k2)}, Soviet Mathematics Doklady,  \textbf{27}  (1983),  372--376.


\bibitem{Nest2}  Y. Nesterov, {\it  Introductory lectures on convex optimization: A basic course}, volume 87 of
Applied Optimization. Kluwer Academic Publishers, Boston, MA, 2004.

\bibitem{Op} Z. Opial, {\it  Weak convergence of the sequence of successive approximations for nonexpansive mappings},  Bull. Amer. Math. Soc.,  \textbf{73 } (1967), 591--597.

\bibitem{PaBo} N.  Parikh,  S. Boyd, {\it Proximal algorithms},  Foundations and trends in optimization, volume 1, (2013), 123--231.

\bibitem{Pey} J. Peypouquet, {\it Convex optimization in normed spaces: theory, methods and examples}, Springer, 2015.

\bibitem{Pey_Sor}  J. Peypouquet, S. Sorin, {\it Evolution equations for maximal monotone operators: asymptotic analysis in continuous and discrete time}, J. Convex Anal, \textbf{17}  (2010), No. 3-4, 1113--1163.

\bibitem{Pol}  B.T. Polyak, {\it Some methods of speeding up the convergence of iteration methods}, U.S.S.R. Comput. Math. Math. Phys., \textbf{4} (1964),  1--17.

\bibitem{Polyak2}{\sc B.T. Polyak}, {\it  Introduction to optimization}. New York: Optimization Software. (1987).
  
  
\bibitem{Rock}  R.T. Rockafellar, {\it Monotone operators and the proximal point algorithm}, SIAM J. Control Optim., \textbf{14} (1976), No. 5, 877--898.


\bibitem{SLB} M. Schmidt, N. Le Roux, F. Bach, {\it  Convergence rates of inexact proximal-gradient methods for convex optimization}, NIPS'11 - 25 th Annual Conference on Neural Information Processing Systems, Dec 2011, Grenada, Spain. (2011) HAL inria-00618152v3.

\bibitem{SDJS} B. Shi, S.  S. Du,  M. I. Jordan,  W. J. Su,
{\it Understanding the acceleration phenomenon via high-resolution differential equations}, Mathematical Programming, 195 (2022), 79--148.
https://doi.org/10.1007/s10107-021-01681-8

\bibitem{SBC}  W. J. Su,  S. Boyd,  E. J. Cand\`es, {\it A differential equation for modeling Nesterov's accelerated gradient method: theory and insights}. Neural Information Processing Systems, \textbf{ 27} (2014),  2510--2518. 

\bibitem{Tikhonov1}   {\sc A.N. Tikhonov}, {\em On the solution of ill-posed problems and the method of regularization}, Soviet Math. Dokl., 4 (1963), 1035-1038 (English translation).

\bibitem{Tikhonov2}   {\sc A.N. Tikhonov}, {\em On the regularization of ill-posed problems}, Soviet Math. Dokl., 4 (1963), 1624-1627 (English translation).

\bibitem{TikhonovLY}  {\sc A.N. Tikhonov, A. Leonov, A.Yagola},  {\em Nonlinear ill-posed problems}, Chapman and Hall, London, 1998.


\bibitem{VSBV} S. Villa, S. Salzo, L. Baldassarres, A. Verri, {\it Accelerated and inexact forward-backward}, SIAM J. Optim.,  \textbf{23}  (2013), No. 3,  1607--1633.

\bibitem{WWJ}
A. Wibisono, A.C. Wilson, M.I.  Jordan, {\it A variational perspective on accelerated methods in optimization}, Proceedings of the National Academy of Sciences, \textbf{113} (2016), No. 47,  E7351--E7358.










 

 




 

 

 

















%

%
%


%






\end{thebibliography}
\end{document}